\documentclass{amsart}
\usepackage{amssymb}
\usepackage{amsmath}
\usepackage{verbatim}

\allowdisplaybreaks

\newcommand{\FF}{\mathcal{F}}
\newcommand{\A}{\mathcal{A}}
\newcommand{\RR}{\mathbb{R}}
\newcommand{\bt}{\bullet}
\newcommand{\Fun}{\mathrm{Fun}}
\newcommand{\gh}{\mathrm{gh}}
\newcommand{\ora}{\overrightarrow}
\newcommand{\ola}{\overleftarrow}
\newcommand{\dd}{\partial}
\newcommand{\hra}{\hookrightarrow}
\newcommand{\be}{\begin{equation}}
\newcommand{\ee}{\end{equation}}
\newcommand{\LL}{\mathcal{L}}

\newcommand{\PP}{\mathcal{P}}

\newcommand{\C}{\mathcal{C}}
\newcommand{\G}{\mathcal{G}}
\newcommand{\V}{\mathcal{V}}

\newcommand{\g}{\mathfrak{g}}
\newcommand{\mr}{\mathrm}

\newcommand{\id}{\mathrm{id}}
\newcommand{\tr}{\mathrm{tr}}
\newcommand{\Str}{\mathrm{Str}}
\newcommand{\Ind}{\mathrm{Ind}}
\newcommand{\im}{\mathrm{im}}
\newcommand{\su}{\mathfrak{su}}
\newcommand{\ra}{\rightarrow}
\newtheorem{proposition}{Proposition}

\newcommand{\ndash}{\nobreakdash-\hspace{0pt}}

\newcommand{\frg}{{\mathfrak{g}}}
\newcommand{\de}{\partial}
\newcommand{\ddd}{{\mathrm{d}}}

\newcommand{\Bl}{\operatorname{Bl}}

\newtheorem{Lem}{Lemma}
\newtheorem{conj}{Conjecture}
\newtheorem{thm}{Theorem}

\theoremstyle{remark}
\newtheorem{Rem}{Remark}

\theoremstyle{definition}
\newtheorem{Def}{Definition}
\newtheorem{Not}{Notation}
\newtheorem{Exa}{Example}

\begin{document}
\title{Remarks on Chern--Simons invariants}
\author{Alberto S. Cattaneo}
\author{Pavel Mn\"ev}
\address{Institut f\"ur Mathematik, Universit\"at Z\"urich--Irchel,
Winterthurerstrasse 190, CH-8057 Z\"urich, Switzerland}
\email{alberto.cattaneo@math.uzh.ch}
\address{Petersburg Department of V. A. Steklov Institute of Mathematics, Fontanka 27, 191023 St. Petersburg, Russia}
\email{pmnev@pdmi.ras.ru}
\begin{abstract}
The perturbative Chern--Simons theory is studied in a finite\ndash dimensional version or assuming that the propagator satisfies certain
properties (as is the case, e.g., with the propagator defined by Axelrod and Singer).
It turns out that the effective BV action
is  a function on cohomology (with shifted degrees)  that solves the quantum master equation and
is defined modulo certain canonical transformations that can be characterized completely. Out of it one obtains invariants.
\end{abstract}
\thanks{This work has been
partially supported by SNF Grant 20-113439, by
the European Union through the FP6 Marie Curie RTN ENIGMA (contract
number MRTN-CT-2004-5652), and by the European Science Foundation
through the MISGAM program. The second author was also supported by RFBR 08-01-00638 and RFBR 09-01-12150-ofi$\_{ }$m grants.}
\maketitle

\setcounter{tocdepth}{3}\tableofcontents
\tableofcontents

\newcommand{\fullref}[1]{\ref{#1} on page~\pageref{#1}}

\section{Introduction}
Since its proposal in Witten's paper \cite{W}, Chern--Simons theory has been a source of fruitful constructions for $3$\ndash manifold invariants.
In the perturbative framework one would like to get the invariants from the Feynman diagrams of the theory. These may be shown to be finite, see \cite{AS}.
However, as in every gauge theory, one has to fix a gauge and then one has to show that the result, the invariant, is independent of the gauge fixing.
In the case when one works around an acyclic connection, this was proved in \cite{AS}, but this assumption rules out the trivial connection.
Gauge-fixing independence for perturbation theory around the trivial connection
of a rational homology sphere
was proved in \cite{A} and, for more general definitions for the propagator,
in \cite{K} and in \cite{BC}. The flexibility in the choice of propagator allows one to show that the invariant is of finite type \cite{KT}.

The case of general $3$\ndash manifolds was not treated in detail, even though the propagators described in \cite{AS,K,BC} are defined in general.
The main point is the presence of zero modes, namely---working around the trivial connection---elements of de~Rham cohomology (with shifted degree) of the manifold
tensor the given Lie algebra.
Out of formal properties of the BV formalism,
it is however clear \cite{L,M0,Co,CF} what the invariant in the general case (of a compact manifold) should be:
a solution of the quantum master equation on the space of zero modes modulo certain BV canonical transformations.
This effective action has already been studied---but only modulo constants---in \cite{Co}, which has
been a source of inspiration for us.


In the first part of this note, we make this precise and mathematically rigorous working with a finite-dimensional version \cite{S} of Chern--Simons theory,
where the algebra of smooth functions on the manifold is replaced by an arbitrary finite-dimensional dg Frobenius algebra (of appropriate degrees).
We are able to produce the solution to the quantum master equation on cohomology and to describe the BV canonical transformations that occur.
Out of this we are able to describe the invariant in the case of cohomology concentrated in degree zero and three (the algebraic version of a rational homology sphere)
and to extract invariants in case the first Betti number is one or, more generally, when the Frobenius algebra is formal.

In the second part we revert to the infinite-dimensional case and show that whatever we did in the finite-dimensional case actually goes through if the propagator
satisfies certain properties. It is good news that the propagator introduced by Axelrod and Singer in \cite{AS} does indeed satisfy them. In particular,
we get an invariant for framed $3$\ndash manifolds as described in Theorem~\fullref{thetheorem}.

The problems with this scheme
are that there is little flexibility in the choice of propagator and that the invariants are defined up to a universal constant that is very difficult to compute.
(Notice that this constant is the same that appears anyway in the case of rational homology spheres in \cite{AS,BC}). For the general case, that is of a propagator
as in \cite{K} or \cite{BC}, we are able to show that all properties but one can easily be achieved. We reduce the last property to Conjecture~\fullref{conj}
which we hope to be able to prove in a forthcoming paper.

During the preparation of this note, we have become aware of independent work by Iacovino \cite{I} on the same topic.

\subsection*{Acknowledgements}
We are grateful to K.~Costello and D.~Sinha 
for insightful discussions.
We also thank C.~Rossi and J.~Stasheff for useful remarks.

\section{Effective BV action}
\label{sec: eff BV action}

Let\footnote{This Section is an adaptation of Section~4.2 from \cite{M}.} $(\FF,\sigma)$ be a finite-dimensional graded vector space endowed with an odd symplectic form $\sigma\in \Lambda^2\FF^*$ of degree -1, which means
$\sigma(u,v)\neq 0 \Rightarrow |u|+|v|=1$ for $u,v\in\FF$. The space of polynomial functions $\Fun(\FF):=S^\bt \FF^*$ is a BV algebra with anti-bracket $\{\bt,\bt\}$ and BV Laplacian\footnote{In the general case, on an odd-symplectic graded manifold, the BV Laplacian $\Delta$ is constructed from the symplectic form and a consistent measure (the ``$SP$-structure''). But here we treat only the linear case, and an odd-symplectic graded vector space $(\FF,\sigma)$ is automatically an $SP$-manifold with Lebesgue measure (the constant Berezinian) $\mu_\FF$.} $\Delta$ generated by the odd symplectic form $\sigma$.

In coordinates: let $\{u_i\}$ be a basis in $\FF$ and $\{x^i\}$ be the dual basis in $\FF^*$. Let us denote
$\sigma_{ij}=\sigma(u_i,u_j)$. Then
\begin{eqnarray*}
\sigma & = & \sum_{i,j} (-1)^{\gh(x^i)}\sigma_{ij}\;\delta x^i \wedge \delta x^j \\
\Delta f & = & \frac{1}{2}\;\sum_{i,j} (-1)^{\gh(x^i)} (\sigma^{-1})^{ij} \frac{\dd}{\dd x^i}\frac{\dd}{\dd x^j}\; f \\
\{f,g\} & = & f \left(\sum_{i,j} (\sigma^{-1})^{ij} \frac{\ola\dd}{\dd x^i} \frac{\ora\dd}{\dd x^j}\right) g
\end{eqnarray*}
In our convention $\sigma_{ij}=-\sigma_{ji}$ and we call grading on $\Fun(\FF)$ the ``ghost number'': $\gh(x^i)=-|u_i|$.

Suppose $(\FF',\sigma')$ is another odd symplectic vector space and $\iota: \FF'\hra\FF$ is an embedding (injective linear map of degree 0) that agrees with the odd symplectic structure:
$$\sigma'(u',v')=\sigma(\iota(u'),\iota(v'))$$
for $u',v'\in \FF'$. Then $\FF$ can be represented as
\be \FF=\iota(\FF')\oplus \FF'' \label{F=F'+F''}\ee
where $\FF'':=\iota(\FF')^\perp$ is the symplectic complement of the image of $\iota$ in $\FF$ with respect to $\sigma$. Hence the algebra of functions on $\FF$ factors
$$\Fun(\FF)\cong \Fun(\FF')\otimes \Fun(\FF'')$$
(the isomorphism depends on the embedding $\iota$). Since (\ref{F=F'+F''}) is an (orthogonal) decomposition of odd symplectic vector spaces, the BV Laplacian also splits:
\be \Delta=\Delta'+\Delta''\label{Delta=Delta'+Delta''}\ee
Here $\Delta'$ is the BV Laplacian on $\Fun(\FF')$, associated to the odd symplectic form $\sigma'$, and $\Delta''$ is the BV Laplacian on $\FF''$ associated to the restricted odd symplectic form $\sigma|_{\FF''}$.

Let $S\in\Fun(\FF)[[\hbar]]$ be a solution to the quantum master equation (QME):
$$\Delta e^{S/\hbar}=0\quad \Leftrightarrow\quad \frac{1}{2}\{S,S\}+\hbar\; \Delta S=0$$
(the BV action on $\FF$).
Let also $\LL\subset\FF''$ be a Lagrangian subspace in $\FF''$. We define the effective (or ``induced'') BV action $S'\in\Fun(\FF')[[\hbar]]$ by the fiber BV integral
\be e^{S'(x')/\hbar}=\int_\LL e^{S(\iota(x')+x'')/\hbar}\mu_\LL \label{BV integral}\ee
where $\mu_\LL$ is the Lebesgue measure on $\LL$ and we use  the notation $x=\sum_i u_i x^i$ for the canonical element of $\FF\otimes\Fun(\FF)$, corresponding to the identity map $\FF\ra\FF$, and analogously for $x'\in \FF'\otimes \Fun(\FF')$, $x''\in \FF''\otimes \Fun(\FF'')$

\begin{proposition} \label{prop1} The Effective action $S'$ defined by (\ref{BV integral}) satisfies the QME on $\FF'$, i.e.\ $$\Delta' e^{S'/\hbar}=0$$
\end{proposition}

\begin{proof} This is a direct consequence of (\ref{Delta=Delta'+Delta''}) and the BV-Stokes theorem (integrals over Lagrangian submanifolds of BV coboundaries vanish):
\begin{multline*}
\Delta' e^{S'(x')/\hbar}= \Delta'\int_\LL e^{S(\iota(x')+x'')/\hbar}\mu_\LL=\int_\LL (\Delta-\Delta'') e^{S(\iota(x')+x'')/\hbar}\mu_\LL=\\
=\int_\LL \Delta e^{S(\iota(x')+x'')/\hbar}\mu_\LL - \int_\LL \Delta'' e^{S(\iota(x')+x'')/\hbar}\mu_\LL=0
\end{multline*}
In the last line the first term vanishes due to QME for $S$, while the second term is zero due to the BV-Stokes theorem.
\end{proof}

The BV integral (\ref{BV integral}) depends on a choice of embedding $\iota:\FF'\hra\FF$ and Lagrangian subspace $\LL\subset \FF''$ (notice that $\FF''=\iota(\FF')^\perp$ itself depends on $\iota$). We call the pair $(\iota,\LL)$ the ``induction data'' in this setting. We are interested in the dependence of the effective BV action $S'$ on deformations of induction data.

Recall that a generic small Lagrangian deformation $\LL_\Psi$ of a Lagrangian submanifold $\LL\subset \FF''$ is given by a gauge fixing fermion $\Psi\in\Fun(\LL)$ of ghost number -1. If $(q^a,p_a)$ are Darboux coordinates on $\FF''$ such that $\LL$ is given by $p=0$, then $\LL_\Psi$ is
$$\LL_\Psi=\left\{(q,p)\; :\; p_a=-\frac{\dd}{\dd q^a}\Psi\right\}\subset\FF''$$
In our case we are interested in linear Lagrangian subspaces and thus only allow quadratic gauge fixing fermions.

A general small deformation of induction data $(\iota,\LL)$ from $(\FF,\sigma)$ to $(\FF',\sigma')$ can be written as
\be \iota\mapsto \iota+\delta\iota_{\perp}+\delta\iota_{||} \quad , \quad \LL\mapsto (\id_\FF-\iota\circ (\delta\iota_{\perp})^T )\LL_\Psi \label{deformation of (iota,L)}\ee
where the ``perpendicular part'' of the deformation of the embedding $\delta\iota_\perp: \FF'\ra \FF$ is a linear map of degree 0 satisfying $\delta\iota_\perp(\FF')\subset \FF''$, while the ``parallel part'' is of the form $\delta\iota_{||}=\iota\circ \delta\phi_{||}$ with $\delta\phi_{||}:\FF'\ra \FF'$ a linear map of degree 0 satisfying $(\delta\phi_{||})^T=-\delta\phi_{||}$ (i.e.\ $\delta\phi_{||}$ lies in the Lie algebra of the group of linear symplectomorphisms $\delta\phi_{||}\in \mr{sp}(\FF',\sigma')$). We use superscript $T$ to denote transposition w.r.t. symplectic structure. Thus it is reasonable to classify small deformations into the three following types:
\begin{itemize}
\item Type I: small Lagrangian deformations of Lagrangian subspace $\LL\mapsto \LL_\Psi$ leaving the embedding $\iota$ intact.
\item Type II: small perpendicular deformations of the embedding $\iota\mapsto \iota+\delta\iota_\perp$ accompanied by
an associated
deformation\footnote{Formula (\ref{type II def}) is explained as follows: one can express the deformation of the embedding as $\iota+\delta\iota_\perp= (\id_\FF+\delta\Phi)\circ \iota$, where $\delta\Phi=\delta\iota_\perp\circ \iota^T-\iota\circ \delta\iota_\perp^T\in \mr{sp}(\FF,\sigma)$, so that $\id_\FF+\delta\Phi$ is an infinitesimal symplectomorphism of the space $\FF$, accounting for the deformation of $\iota$. Then (\ref{type II def}) just means $\LL\mapsto(\id_\FF+\delta\Phi)\circ\LL$.} of the Lagrangian subspace
\be \LL\mapsto (\id_\FF-\iota\circ (\delta\iota_{\perp})^T )\LL \label{type II def}\ee (this is necessary since here we deform the splitting (\ref{F=F'+F''}) and $\LL$ is supposed to be a subspace of the deformed $\FF''$).
\item Type III: small parallel deformations of the embedding $$\iota\mapsto \iota+\delta\iota_{||} = \iota\circ (\id_{\FF'}+\delta\phi_{||})$$ leaving $\LL$ intact.
\end{itemize}
A general small deformation $(\ref{deformation of (iota,L)})$ is the sum of deformations of types I, II, III.

We call the following transformation of the action
\be S\mapsto \tilde{S}=S+\{S,R\}+\hbar \Delta R \label{can transf for S}\ee
(regarded in first order in $R$) the infinitesimal canonical transformation of the action with (infinitesimal) generator $R\in\Fun(\FF)[[\hbar]]$ of ghost number -1. Equivalently
\be e^{S/\hbar}\mapsto e^{\tilde{S}/\hbar}=e^{S/\hbar}+\Delta(e^{S/\hbar}R) \label{can transf, exp form}\ee
The transformed action also solves the QME (in first order in $R$). Infinitesimal canonical transformations generate the equivalence relation on solutions of QME.

\begin{Lem} \label{lm1}
If the action $S$ is changed by an infinitesimal canonical transformation
$S\mapsto S+\{S,R\}+\hbar\Delta R$, then the effective BV action is also changed by an infinitesimal canonical transformation \be S'\mapsto S'+\{S',R'\}'+\hbar \Delta' R' \label{S' transf, lm1}\ee with generator $R'\in\Fun(\FF')[[\hbar]]$ given by fiber BV integral
\be R'=e^{-S'/\hbar}\int_\LL e^{S/\hbar} R\; \mu_\LL \label{R'}\ee
\end{Lem}

\begin{proof}
This follows straightforwardly from (\ref{Delta=Delta'+Delta''}), the BV-Stokes theorem and the exponential form of canonical transformations (\ref{can transf, exp form}):
\begin{multline*} e^{S'/\hbar}\mapsto e^{\tilde{S}'/\hbar} =\int_\LL (e^{S/\hbar}+\Delta (e^{S/\hbar}R)) \mu_\LL =\\
=e^{S'/\hbar}+\Delta' \int_\LL e^{S/\hbar}R\;\mu_\LL+ \underbrace{\int_\LL \Delta''(e^{S/\hbar}R)\;\mu_\LL}_{0}=
e^{S'/\hbar}+\Delta' (e^{S'/\hbar}R')
\end{multline*}
with $R'$ given by (\ref{R'}).
\end{proof}


\begin{proposition}
\label{prop2}
Under general infinitesimal deformation of the induction data $(\iota,\LL)$ as in (\ref{deformation of (iota,L)}) the effective BV action $S'$ is transformed canonically (up to constant shift):
\be S'\mapsto S'+\{S',R'\}'+\hbar \Delta' (R'-R'_{III}) \label{S' transf in prop2}\ee with generator
\be R'=\underbrace{\sigma'(x',\delta\phi_{||}\;x')}_{R'_{III}}+ e^{-S'/\hbar} \int_\LL e^{S/\hbar} (\Psi+\sigma(x'',\delta\iota_\perp\; x')) \;\mu_\LL \label{R' in prop2}\ee
\end{proposition}

\begin{proof}
The deformation (\ref{deformation of (iota,L)}) can be represented in the form
$$\iota\mapsto (\id_\FF+\delta\Phi)\circ \iota\quad, \quad \LL\mapsto (\id_\FF+\delta\Phi)\LL$$
where $\delta\Phi\in\mr{sp}(\FF,\sigma)$ is the infinitesimal symplectomorphism given by
$$\delta\Phi=\{\bt,\Psi\}+(\delta\iota_\perp\circ\iota^T-\iota\circ\delta\iota_\perp^T)+\delta\iota_{||}\circ \iota^T$$
Here $\{\bt,\Psi\}=\exp(\{\bt,\Psi\})-\id_\FF$ is understood as the (infinitesimal) flow generated by Hamiltonian vector field $\{\bt,\Psi\}$ in unit time. The pull-back $(\id_\FF+\delta\Phi)^*: \Fun(\FF)\ra\Fun(\FF)$ acts on functions as canonical transformation
\be f\mapsto f+\{f,R\} \label{can trasf for f}\ee with generator given by
$$R=\frac{1}{2}\sigma(x,\delta\Phi\; x)=\underbrace{\Psi}_{R_I}+\underbrace{\sigma(x,(\delta\iota_\perp\circ\iota^T) x)}_{R_{II}}+\underbrace{\frac{1}{2}\sigma(x,(\delta\iota_{||}\circ \iota^T)x)}_{R_{III}}$$
It is important to note that only the third term (the effect of type III deformation) contributes to  $\Delta R$:
$$\Delta R=\Delta' R= \Delta R_{III}=\frac{1}{2}\;\mr{Str}_{\FF'} \delta\phi_{||}$$
($\Str_{\FF'}$ denotes supertrace over $\FF'$) and $\Delta'' R=0$. The latter implies that $(\id_\FF+\delta\Phi)^* \mu_\LL=\mu_\LL$.
Now we can compute the transformation of the effective BV action $S'$ due to infinitesimal change of induction data:
\begin{multline}
e^{S'/\hbar}\mapsto e^{\tilde{S}'/\hbar}=\int_\LL (\id_\FF+\delta\Phi)^* e^{S/\hbar}\mu_\LL=\\
=e^{-\mr{Str}_{\FF'}\delta\phi_{||}}\int_\LL e^{(S+\{S,R\}+\hbar\Delta R)/\hbar}\mu_\LL=e^{-\mr{Str}_{\FF'}\delta\phi_{||}}e^{(S'+\{S',R'\}'+\hbar\Delta' R')/\hbar}
\label{prop2 eq1}
\end{multline}
Here we used Lemma \ref{lm1} and the generator is given by (\ref{R'}):
$$R'=e^{-S'/\hbar}\int_\LL e^{S/\hbar} R \mu_\LL=e^{-S'/\hbar}\int_\LL e^{S/\hbar} (\Psi+\sigma(x,(\delta\iota_\perp\circ\iota^T) x)+\frac{1}{2}\sigma(x,(\delta\iota_{||}\circ \iota^T)x)) \mu_\LL$$
which yields (\ref{R' in prop2}). At last note that $\frac{1}{2}\;\Str_{\FF'} \delta\phi_{||}=\Delta' \frac{1}{2}\sigma'(x',\delta\phi_{||} x')=\Delta' R'_{III}$ which explains the constant shift in (\ref{S' transf in prop2}).
\end{proof}

\begin{Rem} If we were treating BV actions as log-half-densities (meaning that $e^{S/\hbar}$ is a half-density), we would write an honest canonical transformation (\ref{S' transf, lm1}) instead of (\ref{S' transf in prop2}), with no $-\hbar\Delta'R'_{III}$ shift. This is because the pull-back $(\id_\FF+\delta\Phi)^*$ would then be acting on $S$ by transformation (\ref{can transf for S}) instead of (\ref{can trasf for f}). But in practice one works with effective BV actions defined by a normalized BV integral: if the initial BV action is of the form $S=S_0+S_\mr{int}$ with $S_0$ quadratic in fields (the ``free part'' of BV action) and $S_\mr{int}$ the ``interaction part'', one usually defines
$$e^{S'(x')/\hbar}=\frac{1}{N}\int_\LL e^{S(\iota(x')+x'')/\hbar}\mu_\LL$$
with the normalization factor $N=\int_\LL e^{S_0(x'')/\hbar}\mu_\LL$. The effective action defined via a normalized BV integral is indeed a function rather than log-half-density, and transforms according to (\ref{S' transf in prop2}) under change of induction data.
\end{Rem}

\section{Toy model for effective Chern--Simons theory on zero-modes: effective BV action on cohomology of dg Frobenius algebras}\label{s:toy}
By a dg Frobenius algebra $(\C,d,m,\pi)$ we mean a unital differential graded commutative algebra $\C$ with differential $d: \C^\bt\ra \C^{\bt+1}$ and (super-commutative, associative) product $m: S^2 \C\ra\C$, endowed in addition with a non-degenerate pairing $\pi: S^2\C\ra\RR$ of degree $-k$ (which means $\pi(a,b)\neq 0\Rightarrow |a|+|b|=k$; and $k$ is some fixed integer), satisfying the following consistency conditions:
\begin{eqnarray}
\pi(da,b)+(-1)^{|a|}\pi(a,db)=0 \label{d self-adj} \\
\pi(a,m(b,c))=\pi(m(a,b),c) \label{cyclicity of m}
\end{eqnarray}
for $a,b,c\in \C$.

By a dg Frobenius--Lie\footnote{Alternatively one could use the term ``cyclic dg Lie algebra''.} algebra $(\A,d,l,\pi)$
we mean a differential graded Lie algebra $\A$ with differential $d:\A^\bt\ra\A^{\bt+1}$ and Lie bracket $l: \wedge^2 \A\ra\A$, endowed with a non-degenerate pairing $\pi: S^2\A\ra\RR$ of degree $-k$ satisfying the conditions
\begin{eqnarray}\pi(A,l(B,C))=\pi(l(A,B),C) \label{cyclicity of l}\\
\pi(dA,B)+(-1)^{|A|}\pi(A,dB)=0
\end{eqnarray}
for $A,B,C\in \A$.

If $\g$ is an (ordinary) Lie algebra with non-degenerate ad-invariant inner product $\pi_\g: S^2\g\ra \RR$ (one calls such Lie algebras ``quadratic'') and $(\C,d,m,\pi)$ is a dg Frobenius algebra, then $(\g\otimes\C,d,l,\pi_{\g\otimes\C})$ is a dg Frobenius--Lie algebra. Here one defines $d(X\otimes a):=X\otimes da$,\quad $l(X\otimes a,Y\otimes b):= [X,Y]\otimes m(a,b)$,\quad $\pi_{\g\otimes\C}(X\otimes a, Y\otimes b):=\pi_\g(X,Y)\; \pi(a,b)$. We will usually write $\pi$ instead of $\pi_{\g\otimes\C}$.

\begin{Exa}
If $M$ is a closed (compact, without boundary) orientable smooth manifold of dimension $D$, then de Rham algebra $\Omega^\bt(M)$ is a dg Frobenius algebra with de Rham differential, wedge product and Poincare pairing $\int_M \bt\wedge \bt$ of degree $-D$. If $\g$ is a finite-dimensional Lie algebra with invariant non-degenerate trace $\tr$, then the algebra $\Omega^\bt(M,\g)=\g\otimes \Omega^\bt(M)$ of $\g$-valued differential forms on $M$ is a dg Frobenius--Lie algebra with pairing $\tr\int_M \bt\wedge\bt$ of degree $-D$.
\end{Exa}


\subsection{Abstract Chern--Simons action from a dg Frobenius algebra}
Let $(\C,d,m,\pi)$ be a finite dimensional non-negatively graded dg Frobenius algebra with pairing $\pi$ of degree -3
$$\C=\C^0\oplus \C^1[-1] \oplus \C^2[-2]\oplus \C^3[-3]$$
We denote by $B_i=\dim H^i (\C)$ the Betti numbers. Due to non-degeneracy of $\pi$, there is an isomorphism $\pi_\flat: \C^\bt\cong (\C^{3-\bt})^*$ (the Poincare duality). Induced pairing on cohomology is also automatically non-degenerate, and so Poincare duality descends to cohomology: $\pi_\flat: H^\bt(\C)\cong (H^{3-\bt}(\C))^*$. Hence $B_0=B_3$, $B_1=B_2$. We will suppose in addition that $B_0=B_3=1$ (so that $\C$ models the de Rham algebra of a connected manifold).

Let $\g$ be a finite dimensional quadratic Lie algebra of coefficients and let $(\g\otimes \C,d,l,\pi)$ be the corresponding dg Frobenius--Lie structure on $\g\otimes\C$. Then we can construct an odd symplectic space of BV fields
$$\FF=\g\otimes\C[1]$$
with odd symplectic structure of degree -1 given by
\be\sigma(sA,sB)=(-1)^{|A|}\pi(A,B)\label{sigma from pi}\ee
Here $s:\g\otimes\C\ra \g\otimes\C[1]$ is the suspension map.
Let us also introduce the notation $\omega$ for the canonical element of $(\g\otimes\C)\otimes\Fun(\FF)$ corresponding to the desuspension map $s^{-1}:\FF\ra \g\otimes\C$. If $\{e_I\}$ is a basis in $\C$ and $\{T_a\}$ is an orthonormal basis in $\g$, then we can write
$$\omega=\sum_{I,a}T_a e_I\; \omega^{Ia}$$
where $\{\omega^{Ia}\}$ are the corresponding coordinates on $\FF$. By abuse of terminology we call $\omega$ the ``BV field''. Let us introduce notations for the structure constants:
$\pi_{IJ}=\pi(e_I,e_J)$, $m_{IJK}=\pi(e_I,m(e_J,e_K))$, $f_{abc}=\pi_\g(T_a, [T_b,T_c])$, $d_{IJ}=\pi(e_I,d e_J)$. We will also use the shorthand notation for degrees $|I|=|e_I|$. In terms of $\pi$ the BV Laplacian and the anti-bracket are
\begin{eqnarray*}
\Delta f & = & \frac{1}{2}\;\sum_{I,J,a}(\pi^{-1})^{IJ}\frac{\dd}{\dd\omega^{Ia}}\frac{\dd}{\dd\omega^{Ja}} f\\
\{f,g\} & = & f \left(\sum_{I,J,a} (-1)^{|I|+1}(\pi^{-1})^{IJ} \frac{\ola\dd}{\dd \omega^{Ia}} \frac{\ora\dd}{\dd \omega^{Ja}}\right) g
\end{eqnarray*}

\begin{proposition} \label{prop3}
The action $S\in\Fun(\FF)$ defined as
\begin{multline} S:=\frac{1}{2}\pi(\omega,d\omega)+\frac{1}{6}\pi(\omega,l(\omega,\omega))=\\
=\frac{1}{2}\sum_{I,J,a} (-1)^{|I|+1} d_{IJ}\omega^{Ia}\omega^{Ja}+\frac{1}{6}\sum_{I,J,K,a,b,c} (-1)^{|J|\,(|K|+1)} f_{abc} m_{IJK} \omega^{Ia} \omega^{Jb} \omega^{Kc}  \label{abstract CS action}
\end{multline}
satisfies the QME with BV Laplacian defined by the odd symplectic structure (\ref{sigma from pi}) on $\FF$.
\end{proposition}

\begin{proof}
Indeed, let us check the CME:
\begin{multline*}
\frac{1}{2}\{S,S\}=\\
\frac{1}{8}\{\pi(\omega,d\omega),\pi(\omega,d\omega)\}+\frac{1}{12}\{\pi(\omega,d\omega),\pi(\omega,l(\omega,\omega))\}+\frac{1}{72}\{\pi(\omega,l(\omega,\omega)),\pi(\omega,l(\omega,\omega))\}\\
=-\frac{1}{2}\pi(\omega,d^2\omega)-\frac{1}{2}\pi(\omega,dl(\omega,\omega))-\frac{1}{8}\pi(\omega,l(\omega,l(\omega,\omega)))=0
\end{multline*}
The first term vanishes due to $d^2=0$, the second --- due to the Leibniz identity for $\g\otimes\C$, since property (\ref{cyclicity of l}) implies
\[
\frac{1}{2}\pi(\omega,dl(\omega,\omega))=\frac{1}{6}\pi\left(\omega,dl(\omega,\omega)-l(d\omega,\omega)+l(\omega,d\omega)\right),
\]
and the third term is zero due to the Jacobi identity for $\g\otimes\C$. Next, check the quantum part of the QME:
$$\hbar\;\Delta S = -\hbar\;\frac{1}{2}\; \Str_{\g\otimes\C}\; d - \hbar\;\frac{1}{2}\;\Str_{\g\otimes\C}\;l(\omega,\bt)=0$$
Here the first term vanishes since $d$ raises degree and the second term vanishes due to unimodularity of Lie algebra $\g$.
\end{proof}

The BV action (\ref{abstract CS action}) can be viewed as an abstract model (or toy model, since $\C$ is finite dimensional) for Chern--Simons theory on a connected closed orientable 3-manifold. We associate such a model to any finite dimensional non-negatively graded dg Frobenius algebra $\C$ with pairing of degree $-3$ and $B_0=B_3=1$, and arbitrary finite dimensional quadratic Lie algebra of coefficients $\g$. We are interested in the effective BV action for (\ref{abstract CS action}) induced on cohomology $\FF'=H^\bt(\C,\g)[1]$. We will now specialize the general induction procedure sketched in Section~\ref{sec: eff BV action} to this case.

\subsection{Effective action on cohomology}\label{sec 2.2}
Let $\iota: H^\bt(\C)\hra \C^\bt$ be an embedding of cohomology into $\C$. Note that $\iota$ is not just an arbitrary chain map between two fixed complexes, but is also subject to condition $\iota ([a])=a+d(\ldots)$ for any cocycle $a\in\C$. This implies in particular that the only allowed deformations of $\iota$ are of the form $\iota\mapsto \iota+ d\; \delta I$ where
$\delta I: H^\bt(\C)\ra \C^{\bt-1}$
is an arbitrary degree -1 linear map. This is indeed a type II deformation (in the terminology of
Section~\ref{sec: eff BV action}), while type III deformations are prohibited in this setting.

Let also $K: \C^\bt\ra \C^{\bt-1}$ be a symmetric chain homotopy retracting $\C^\bt$ to $H^\bt(\C)$, that is a degree -1 linear map satisfying
\begin{eqnarray}
dK+Kd=\id_\C-\PP' \label{dK+Kd}\\
\quad \pi(Ka,b)+(-1)^{|a|}\pi(a,Kb)=0 \label{K self-adj}\\
K\circ\iota=0\label{K iota=0}
\end{eqnarray}
where $\PP'=\iota\circ \iota^T: \C\ra \C$ is the orthogonal (w.r.t. $\pi$) projection to the representatives of cohomology in $\C$. We require the additional property
\be K^2=0 \label{K^2=0}\ee

\begin{Rem}[cf. \cite{GL}] An arbitrary linear map $K_0: \C^\bt\ra \C^{\bt-1}$ satisfying just (\ref{dK+Kd}) can be transformed into a chain homotopy $K$ with all the properties (\ref{dK+Kd},\ref{K self-adj},\ref{K iota=0},\ref{K^2=0}) via a chain of transformations $K_0\mapsto K_1\mapsto K_2\mapsto K_3=K$ where
\begin{eqnarray}
K_1 &=& \frac{1}{2}(K_0-K_0^T) \nonumber \\
K_2 &=& (\id_\C-\PP')\; K_1\; (\id_\C-\PP') \label{K_2} \\
K_3 &=& K_2\; d\; K_2 \label{K_3=KdK}
\end{eqnarray}
\end{Rem}

Having $\iota$ and $K$ we can define a Hodge decomposition for $\C$ into representatives of cohomology, $d$-exact part and $K$-exact part:
\be \C=\im(\iota)\oplus \underbrace{\C_{d-ex}}_{\mr{im}(d)}\oplus \underbrace{\C_{K-ex}}_{\mr{im}(K)} \label{Hodge decomp for C}\ee
Properties (\ref{K self-adj}), (\ref{K iota=0}), (\ref{K^2=0}) and skew-symmetry of differential (\ref{d self-adj}) imply the orthogonality properties for Hodge decomposition (\ref{Hodge decomp for C}):
$$\im(\iota)^\perp=\C_{d-ex}\oplus \C_{K-ex} \;,\quad (\C_{d-ex})^\perp=\im(\iota)\oplus \C_{d-ex} \;,\quad (\C_{K-ex})^\perp = \im(\iota)\oplus \C_{K-ex} $$

In terms of Hodge decomposition (\ref{Hodge decomp for C}) the splitting (\ref{F=F'+F''}) of the space of BV fields $\FF=\g\otimes\C[1]$ is given by
$$\underbrace{\g\otimes\C[1]}_{\FF}=\underbrace{\iota(H^\bt(\C,\g)[1])}_{\iota(\FF')}\oplus \underbrace{\g\otimes \C_{d-ex}[1]\oplus \g\otimes \C_{K-ex}[1]}_{\FF''}$$
and we choose the Lagrangian subspace
\be \LL_K=\g\otimes\C_{K-ex}[1]\subset \FF'' \label{L_K}\ee
We define the ``effective BV action on cohomology'' (or ``on zero-modes'') $W\in\Fun(\FF')[[\hbar]]$ for an abstract Chern--Simons action (\ref{abstract CS action}) by a normalized fiber BV integral
\be e^{W(\alpha)/\hbar}=\frac{1}{N}\int_{\LL_K}e^{S(\iota(\alpha)+\omega'')/\hbar}\mu_{\LL_K} \label{W def by BV integral}\ee
where \be N=\int_{\LL_K} e^{S_0(\omega'')/\hbar} \mu_{\LL_K} \label{N}\ee is the normalization factor
and $$S_0(\omega'')=\frac{1}{2}\pi(\omega'',d\omega'')$$
is the free part of the action $S$.
To lighten somewhat the notation, we denoted the effective action by $W$ instead of $S'$ and the BV field associated to $\FF'=H^\bt(\C,\g)[1]$ by $\alpha$ instead of $\omega'$. Let $\{e_p\}$ be a basis of the cohomology $H^\bt(\C)$. Then $\alpha=\sum_{a,p} T_a e_p \alpha^{pa}$ where $\alpha^{pa}$ are coordinates on $\FF'$ with ghost numbers $\gh(\alpha^{pa})=1-|e_p|$.  We have the following decomposition of $S(\iota(\alpha)+\omega'')$:
\begin{multline}
S(\iota(\alpha)+\omega'')=\underbrace{\frac{1}{6}\pi(\iota(\alpha),l(\iota(\alpha),\iota(\alpha)))}_{W_\mr{prod}(\alpha)}+\underbrace{\frac{1}{2}\pi(\omega'',d\omega'')}_{S_0(\omega'')}+\\
+\underbrace{\frac{1}{2}\pi(\omega'',l(\iota(\alpha),\iota(\alpha)))+ \frac{1}{2}\pi(\iota(\alpha),l(\omega'',\omega''))+\frac{1}{6}\pi(\omega'',l(\omega'',\omega''))}_{S_\mr{int}(\alpha,\omega'')}
\label{S=W_prod+S_0+S_int}
\end{multline}

The perturbation expansion for (\ref{W def by BV integral}) is obtained in a standard way and can be written as
\begin{multline}W(\alpha)=W_\mr{prod}(\alpha)+\hbar\log\left(e^{-\hbar\;\frac{1}{2}\pi^{-1}(\frac{\dd}{\dd\omega''},K\frac{\dd}{\dd\omega''})}|_{\omega''=0}\circ e^{S_\mr{int}(\alpha,\omega'')/\hbar}\right)=\\
=\sum_{l=0}^\infty \hbar^l \sum_{n=0}^\infty \sum_{\Gamma\in G_{l,n}} \frac{1}{|\mr{Aut}(\Gamma)|} W_\Gamma(\alpha)
\label{W pert series}
\end{multline}
Where $G_{l,n}$ denotes the set of connected non-oriented Feynman graphs with vertices of valence 1 and 3 (we would like to understand them as trivalent graphs with ``leaves'' allowed, i.e.\ external edges), with $l$ loops and $n$ leaves. The contribution $W_\Gamma(\alpha)$ of each Feynman graph $\Gamma\in G_{l,n}$ is a homogeneous polynomial of degree $n$ in $\{\alpha^{pa}\}$ and of ghost number 0 obtained by decorating each leaf of $\Gamma$ by $\iota^I_p \alpha^{pa}$, each trivalent vertex by $f_{abc}m_{IJK}$ and each (internal) edge by $\delta^{ab}K^{IJ}$, and taking contraction of all indices, corresponding to incidence of vertices and edges in $\Gamma$. One should also take into account signs for contributions, which can be obtained from the exponential formula for perturbation series (\ref{W pert series}). The cubic term
\begin{multline} W_\mr{prod}(\alpha)=\frac{1}{6}\pi(\iota(\alpha),l(\iota(\alpha),\iota(\alpha)))=\\
=\frac{1}{6}\sum_{a,b,c}\sum_{p,q,r} (-1)^{|e_q|\, (|e_r|+1)} f_{abc} \mu_{pqr} \alpha^{pa} \alpha^{qb} \alpha^{rc} \label{W_prod}\end{multline}
is the contribution of the simplest Feynman diagram $\Gamma_{0,3}$, the only element of $G_{0,3}$. Here $\mu_{pqr}=\pi(\iota(e_p),m(\iota(e_q),\iota(e_r)))$ are structure constants of the induced associative product on $H^\bt(\C)$ (hence the notation $W_\mr{prod}$).

\begin{Rem} The perturbative expansion (\ref{W pert series}) is related to homological perturbation theory (HPT) in the following way. Denote by $c_{l,n}(\alpha)=n!\sum_{\Gamma\in G_{l,n}} \frac{1}{|\mr{Aut}(\Gamma)|} W_\Gamma(\alpha)$ the total contribution of Feynman graphs with $l$ loops and $n$ leaves to the effective action (\ref{W pert series}), with additional factor $n!$. Then each
\[
c_{l,n}\in S^n(\FF'^*)\cong \mr{Hom}(\wedge^n (H^\bt(\C,\g)),\RR)
\]
can be understood as a (super-)anti-symmetric $n$-ary operation on cohomology $H^\bt(\C,\g)$, taking values in numbers. Now suppose $l_n\in \mr{Hom}(\wedge^n (H^\bt(\C,\g)),H^\bt(\C,\g))$ are the $L_\infty$ operations on cohomology, induced from dg Lie algebra $\g\otimes\C$ (by means of HPT). Then it is easy to see that $c_{0,n+1}(\alpha_0,\ldots,\alpha_n)=\pi'(\alpha_0,l_n(\alpha_1,\ldots,\alpha_n))$ and trees for HPT (the Lie version of trees from \cite{KS}, cf. also \cite{GL}) are obtained from Feynman trees for $W(\alpha)$ by assigning one leaf as a root and inserting the inverse of pairing\footnote{We use notation $\pi'=\pi(\iota(\bt),\iota(\bt)): \C\otimes\C\ra\RR$ for the induced pairing on cohomology $H^\bt(\C)$. By a slight abuse of notation, we also use $\pi'$ to denote the pairing on $H^\bt(\C,\g)$ induced from $\pi_{\g\otimes\C}$.} $(\pi')^{-1}$ there, or vice versa: Feynman trees are obtained from trees of HPT by reverting the root\footnote{This means the following: let $T$ be a binary rooted tree with $n$ leaves, oriented towards the root; let $\bar{T}$ be the non-oriented (and non-rooted) tree with $n+1$ leaves, obtained from $T$ by forgetting the orientation (and treating the root as additional leaf). Then the weight $W_{\bar{T}}(\alpha)$ of $\bar{T}$ as a Feynman graph and the contribution $l_{T}$ of tree $T$ (without the symmetry factor) to the induced $L_\infty$ operation $l_n$ are related by $W_{\bar{T}}(\alpha)=\pi'(\alpha,l_T(\underbrace{\alpha,\ldots,\alpha}_n))$. Pictorially this is represented by inserting a bivalent vertex (associated to the operation $\pi'(\bt,\bt)$) at the root of $T$. Both edges incident to this vertex are incoming, thus we say that the root becomes reverted.} with $\pi'$ and forgetting the orientation of edges (cf.\ Section~7.2.1 of \cite{M}).
Thus we can loosely say that the BV integral (\ref{W def by BV integral}) defines a sort of ``loop enhancement'' of HPT for a cyclic dg Lie algebra $\g\otimes\C$. Also, in this language (due to A.~Losev), using the BV-Stokes theorem to prove that the effective action $W$ satisfies the quantum master equation can be viewed as the loop-enhanced version of using the HPT machinery to prove the system of quadratic relations (homotopy Jacobi identities) on induced $L_\infty$ operations $l_n$.
\end{Rem}

Let us introduce a Darboux basis in $H^\bt(\C)$. Namely, let $e_{(0)}=[1]$ be the basis vector in $H^0(\C)$, the cohomology class of unit $1\in \C^0$ (recall that we assume $B_0=\dim H^0(\C)=1$) and let $e_{(3)}$ be the basis vector in $H^3(\C)$, satisfying $\pi'(e_{(0)},e_{(3)})=1$ (i.e.\ $e_{(3)}$ is represented by some top-degree element $v=\iota(e_{(3)})\in \C^3$, normalized by the condition $\pi(1,v)=1$). Let also $\{e_{(1)i}\}$ be some basis in $H^1(\C)$ and $\{e_{(2)}^i\}$ the dual basis in $H^2(\C)$, so that $\pi'(e_{(1)i},e_{(2)}^j)=\delta^j_i$. The BV field $\alpha$ is then represented as
$$\alpha=\sum_a e_{(0)}T_a \alpha_{(0)}^a+ \sum_{a,i} e_{(1)i}T_a \alpha_{(1)}^{ia}+ \sum_{a,i} e_{(2)}^i T_a \alpha_{(2)i}^{a}+\sum_a e_{(3)}T_a \alpha_{(3)}^a$$
In Darboux coordinates $\{\alpha_{0}^a,\alpha_{(1)}^{ia},\alpha_{(2)i}^a,\alpha_{(3)}^a\}$ the BV Laplacian on $\FF'$ is
$$\Delta'=\sum_a \frac{\dd}{\dd \alpha_{(0)}^a}\;\frac{\dd}{\dd \alpha_{(3)}^a}+\sum_{a,i} \frac{\dd}{\dd \alpha_{(1)}^{ia}}\;\frac{\dd}{\dd \alpha_{(2)i}^a}$$
It is also convenient to introduce $\g$-valued coordinates on $\FF'$:
$$\alpha_{(0)}=\sum_a T_a \alpha_{(0)}^a \;,\; \alpha_{(1)}^i=\sum_a T_a \alpha_{(1)}^{ia} \;,\; \alpha_{(2)i}=\sum_a T_a \alpha_{(2)i}^{a} \;,\; \alpha_{(3)}=\sum_a T_a \alpha_{(3)}^a$$
The ghost numbers of $\alpha_{(0)},\alpha_{(1)}^i,\alpha_{(2)i},\alpha_{(3)}$ are 1, 0, -1, -2 respectively.
In terms of this Darboux basis, the trivial part (\ref{W_prod}) of the effective action is
\begin{multline*}
W_\mr{prod}(\alpha)=\\
=\frac{1}{2}\sum_{a,b,c}f_{abc}\alpha_{(0)}^a\alpha_{(0)}^b\alpha_{(3)}^c-\sum_{a,b,c}\sum_i f_{abc}\alpha_{(0)}^a\alpha_{(1)}^{ib}\alpha_{(2)i}^c+
\frac{1}{6} \sum_{a,b,c}\sum_{i,j,k} f_{abc} \mu_{ijk} \alpha_{(1)}^{ia} \alpha_{(1)}^{jb} \alpha_{(1)}^{kc} =\\
=\underbrace{\frac{1}{2}\pi_\g (\alpha_{(3)}, [\alpha_{(0)},\alpha_{(0)}])}_{W_\mr{prod}^{003}}\underbrace{-\sum_i \pi_\g (\alpha_{(0)},[\alpha_{1}^i,\alpha_{(2)i}])}_{W_\mr{prod}^{012}} +\underbrace{\frac{1}{6}\sum_{i,j,k} \mu_{ijk}\; \pi_g (\alpha_{(1)}^i, [\alpha_{(1)}^j,\alpha_{(1)}^k])}_{W_\mr{prod}^{111}}
\end{multline*}
Here $\mu_{ijk}$ is the totally antisymmetric tensor of structure constants of multiplication of 1-cohomologies: $\mu_{ijk}=\pi(\iota(e_{(1)i}),m(\iota(e_{(1)j}),\iota(e_{(1)k})))$.

\begin{proposition}
\label{prop4}
The effective BV action $W$, induced from the abstract Chern--Simons action (\ref{abstract CS action}) on $\FF'=H^\bt(\C,\g)[1]$ has the form
\be W(\alpha)=W_\mr{prod}(\alpha)+F(\alpha_{(1)}^1,\ldots,\alpha_{(1)}^{B_1};\hbar) \label{W ansatz}\ee
where $F\in \left(\Fun(\g^{B_1})[[\hbar]]\right)^\g$ is some $\hbar$-dependent function on $H^1(\C,\g)[1]\cong \g^{B_1}$, invariant under the diagonal adjoint action of $\g$, i.e.
\be F(\alpha_{(1)}^1+[X,\alpha_{(1)}^1],\ldots,\alpha_{(1)}^{B_1}+[X,\alpha_{(1)}^{B_1}];\hbar)=F(\alpha_{(1)}^1,\ldots,\alpha_{(1)}^{B_1};\hbar) \mod X^2
\label{F ad-invariance}\ee
at the first order in $X\in\g$.
\end{proposition}

\begin{proof}
Ansatz (\ref{W ansatz}) follows from the observation that the values of individual Feynman graphs $\Gamma\neq \Gamma_{0,3}$ in (\ref{W pert series}) depend only on
the 1-cohomology: $W_\Gamma=W_\Gamma(\{\alpha_{(1)}^i\})$. The argument is as follows: suppose not all leaves of $\Gamma$ are decorated with insertions of 1-cohomology. Then, since $\gh(W_\Gamma)=0$, there is at least one leaf decorated with insertion of 0-cohomology. Since $\iota(H^0(\C))=\RR\cdot 1$, the value of $\Gamma$ will contain one of the expressions $K^2(\cdots)$, $K(\iota(\cdots))$. Hence Feynman diagrams with insertion of cohomology of degree $\neq 1$ vanish due to (\ref{K iota=0}), (\ref{K^2=0}). This proves (\ref{W ansatz}).

By construction, $W$ has to satisfy the QME (Proposition \ref{prop1}). The QME for an action satifying ansatz (\ref{W ansatz}) is equivalent to ad-invariance of $F$ (\ref{F ad-invariance}):
\begin{multline*}
\frac{1}{2}\{W,W\}'+\hbar\Delta'W=\{W_\mr{prod}^{012},F\}'=\sum_{i}\sum_{a,b,c} f_{abc}\alpha_{(0)}^a\alpha_{(1)}^{ib}\frac{\dd}{\dd\alpha_{(1)}^{ic}} F=\\
=\sum_i \left<[\alpha_{(0)},\alpha_{(1)}^i],\frac{\dd}{\dd\alpha_{(1)}^i}\right>_\g F
\end{multline*}
where $<\bt,\bt>_\g$ denotes the canonical pairing between $\g$ and $\g^*$.

Another explanation of (\ref{F ad-invariance}) is the following: if $\g$ is the Lie algebra of the Lie group $\G$, then the original abstract Chern--Simons action (\ref{abstract CS action}) is invariant under the adjoint action of $\G$, i.e.\ $\omega\mapsto g\omega g^{-1}$ for $g\in \G$. The embedding $\iota$ and the choice of the Lagrangian subspace $\LL\subset \FF''$ are also compatible with this symmetry. Hence $W(\alpha)$ is also invariant under the adjoint action of $\G$: $\alpha\mapsto g\alpha g^{-1}$, and (\ref{F ad-invariance}) is the infinitesimal form of this symmetry.
\end{proof}

\begin{Rem} There is another argument for ansatz (\ref{W ansatz}) that can be formulated on the level of BV integral (\ref{W def by BV integral}) itself, rather than on
the level of Feynman diagrams. Namely, the restriction of $S_\mr{int}(\alpha,\omega'')$ (we refer to decomposition (\ref{S=W_prod+S_0+S_int})) to the Lagrangian subspace $\LL_K$ does not depend on $\alpha_{(0)}$. This means that the only term depending on $\alpha_{(0)}$ in (\ref{S=W_prod+S_0+S_int}) is the trivial one $W_\mr{prod}(\alpha)$, constant on $\LL_K$. Hence the non-trivial part of the effective action $W(\alpha)-W_\mr{prod}(\alpha)$ does not depend on $\alpha_{(0)}$. Since it also has to be of ghost number zero, it can only depend on $\alpha_{(1)}$.
\end{Rem}

\subsection{Dependence of the effective action on cohomology on induction data} \label{sec 2.3}
The effective action $W(\alpha)$ depends on induction data $(\iota,K)$, and we are interested in describing how $W(\alpha)$ changes due to the
deformation of $(\iota,K)$.

In the terminology of Section~\ref{sec: eff BV action},
the type I deformations of induction data are of the form $\iota\mapsto \iota$, $K\mapsto K+[d,\delta\kappa]$ where
$\delta\kappa: \C_{d-ex}^\bt\ra \C_{K-ex}^{\bt-2}$ is a skew-symmetric linear map of degree -2. The corresponding deformation of the Lagrangian subspace $\LL_K$ is described by the gauge fixing fermion $\Psi=\frac{1}{2}\pi(\omega,d\;\delta\kappa\;d\; \omega)$. Type II deformations change embedding as $\iota\mapsto \iota+d\; \delta I$ where $\delta I: H^\bt(\C)\ra \C_{K-ex}^{\bt-1}$, and change chain homotopy in a minimal way (so as not to spoil properties (\ref{K iota=0}), (\ref{dK+Kd})): $K\mapsto K+\iota\; \delta I^T-\delta I\; \iota^T$. Type III transformations are forbidden in this setting, as discussed above (we have a canonical surjective map from $\ker d\subset \C$ to $H^\bt(\C)$ that sends the cocycle $\alpha$ to its cohomology class $[\alpha]$).

Due to Proposition \ref{prop2}, an infinitesimal deformation of $(\iota,K)$ induces an infinitesimal canonical transformation of $W(\alpha)$
$$W\mapsto W+\{W,R'\}'+\hbar \Delta' R'$$
with generator given by the fiber BV integral
\be R'(\alpha)=e^{-W(\alpha)/\hbar}\int_{\LL_K}e^{S(\iota(\alpha)+\omega'')/\hbar} \left(\frac{1}{2}\pi(\omega'',d\;\delta\kappa \;d \omega'')+\pi(\omega'',d\; \delta I \alpha) \right) \mu_{\LL_K} \label{R' via BV integral}\ee
This integral is evaluated perturbatively in analogy with (\ref{W pert series}):
\begin{multline}
R'(\alpha)=\\
e^{-W(\alpha)/\hbar}\cdot \; \left.e^{-\hbar\;\frac{1}{2}\pi^{-1}(\frac{\dd}{\dd\omega''},K\frac{\dd}{\dd\omega''})}\right|_{\omega''=0}\circ\\
\circ \left(e^{S_\mr{int}(\alpha,\omega'')/\hbar}\cdot (\frac{1}{2}\pi(\omega'',d\;\delta\kappa \;d \omega'')+\pi(\omega'',d\; \delta I \alpha) )\right)=\\
\sum_{l=0}^\infty \hbar^l \sum_{n=0}^\infty \sum_{\Gamma^M\in G^M_{l,n}} \frac{1}{|\mr{Aut}(\Gamma^M)|} R'_{\Gamma^M}(\alpha)
\label{R' pert series}
\end{multline}
Here the superscript $M$ stands for ``marked edge'', $G^M_{l,n}$ is the set of connected non-oriented trivalent graphs 
with $l$ loops and $n$ leaves and either one leaf or one internal edge marked. Values of Feynman graphs $R'_{\Gamma^M}(\alpha)$ are now homogeneous polynomials of degree $n$ and ghost number -1 on $\FF'$, obtained by the same Feynman rules as for $W_\Gamma(\alpha)$, supplemented with a Feynman rule for marked edge: we decorate the marked leaf with $\delta I^I_p \alpha^{pa}$ and marked internal edge with $\delta^{ab}\delta\kappa^{IJ}$.

\begin{proposition}\label{prop5}
The generator of the the infinitesimal canonical transformation induced on the effective BV action $W(\alpha)$ by the infinitesimal change of induction data $(\iota,K)$ has the following form:
\be R'(\alpha)=\sum_{a,i}\alpha_{(2)i}^a G^{ia}(\alpha_{(1)}^1,\ldots,\alpha_{(1)}^{B_1};\hbar) \label{prop5 R'}\ee
where $G=\sum_{a,i} G^{ia}\frac{\dd}{\dd\alpha_{(1)}^{ia}}\in(\mr{Vect}(\g^{B_1})[[\hbar]])^\g$ is some $\hbar$-dependent vector field on $H^1(\C,\g)[1]\cong\g^{B_1}$, equivariant under the diagonal adjoint action of $\g$, i.e.\
\be G^{i}(\alpha_{(1)}^1+[X,\alpha_{(1)}^1],\ldots,\alpha_{(1)}^{B_1}+[X,\alpha_{(1)}^{B_1}];\hbar)=
[X,G^{i}(\alpha_{(1)}^1,\ldots,\alpha_{(1)}^{B_1};\hbar)]\mod X^2 \label{prop5 G equivar}\ee
at first order in $X\in\g$, for all $i=1,\ldots,B_1$ (and we set $G^i:=\sum_a T_a G^{ia}$). The canonical transformation with generator (\ref{prop5 R'}) in terms of ansatz (\ref{W ansatz}) is
\be F\mapsto F+G\circ (W_\mr{prod}^{111}+F) +\hbar\;\mr{div}(G) \label{prop5 G acts on F}\ee
\end{proposition}

\begin{proof}
The argumentat for ansatz (\ref{prop5 R'}) is pretty much the same as for (\ref{W ansatz}): the value $R'_{\Gamma^M}(\alpha)$ of each Feynman graph $\Gamma^M\in G^M_{l,n}$ is linear in $\alpha_{(2)}$ and does not depend on $\alpha_{(0)}, \alpha_{(3)}$ for the following reason: Unless we decorate one leaf of $\Gamma^M$ by $\alpha_{(2)}$ and all other leaves by $\alpha_{(1)}$, some leaf has to be decorated by $\alpha_{(0)}$ (since the total ghost number of $R'_{\Gamma^M}(\alpha)$ has to be -1). Then the contribution of this decoration of $\Gamma^M$ vanishes due to $\iota(H^0(\C))=\RR\cdot 1$ and the vanishing of the expressions $\delta I (e_{(0)}),\;K^2,\; K\;\delta\kappa,\;\delta\kappa\; K,\;K\iota, \;\delta\kappa\;\iota$, one of which necessarily appears as contribution of a neighborhood of the place of insertion of 0-cohomology on the Feynman graph. This proves ansatz (\ref{prop5 R'}).

The equivariance of $G$ (\ref{prop5 G equivar}) is equivalent to the fact that a canonical transformation with generator (\ref{prop5 R'}) preserves ansatz (\ref{W ansatz}) for $W(\alpha)$. Indeed, if $G$ were not equivariant, the term $\{W_\mr{prod}^{012},R'\}'$ would produce $\alpha_{(0)}$-dependence for the canonically transformed effective action. The other explanation is that the equivariance of $G$ is due to the invariance of $R'$ under the adjoint action of the group $\G$, which is due to the fact that a deformation of $(\iota,K)$ is trivial in $\g$-coefficients and hence consistent with the adjoint $\G$-action.

Rewriting the canonical transformation of the effective action (\ref{W ansatz}) with generator (\ref{prop5 R'}) as (\ref{prop5 G acts on F}) is straightforward.
\end{proof}

\begin{Rem} Analogously to Proposition \ref{prop4}, one can prove ansatz (\ref{prop5 R'}) on the level of BV integral (\ref{R' via BV integral}) instead of using Feynman diagrams. Namely, expressions $S(\iota(\alpha)+\omega'')-W_\mr{prod}(\alpha)$,\; $W(\alpha)-W_\mr{prod}(\alpha)$ and the expression in parentheses in (\ref{R' via BV integral}) all do not depend on $\alpha_{(0)}$. Hence $R'$ does not depend on $\alpha_{(0)}$. But it also has to be of ghost number -1, which can only be achieved if it is of form (\ref{prop5 R'}).
\end{Rem}

\subsection{Invariants} \label{sec 2.4}
We are interested in describing the effective action $W(\alpha)$ on cohomology modulo changes of induction data $(\iota,K)$. Due to Propositions \ref{prop4}, \ref{prop5}, we can give a complete solution (i.e.\ describe the complete invariant) in case $B_1=0$, and find some partial solution (i.e.\ describe some, probably incomplete, invariant) for the case of a formal Frobenius algebra $\C$, meaning that we can find representatives for cohomology $\iota:H^\bt(\C)\hra \C$ closed under multiplication. In particular, in case $B_1=1$ the algebra $\C$ is necessarily formal.

\begin{proposition} \label{prop6}
If $B_1=0$, the effective action on cohomology is
\be W(\alpha)=W_\mr{prod}^{003}(\alpha)+F(\hbar) \label{prop6 W}\ee
where $W_\mr{prod}^{003}(\alpha)=\frac{1}{2}\pi_\g(\alpha_{(3)},[\alpha_{(0)},\alpha_{(0)}])$ and $F(\hbar)$ is an $\hbar$-dependent constant, invariant under deformations of induction data $(\iota,K)$.
\end{proposition}

\begin{proof}
Ansatz (\ref{prop6 W}) is a restriction of (\ref{W ansatz}) to the case $B_1=0$. Due to (\ref{prop5 R'}) and $B_1=0$, the generator of induced canonical transformation necessarily vanishes $R'=0$. Hence $F(\hbar)$ is invariant under deformation of $(\iota,K)$.
\end{proof}

So $F(\hbar)$ is the complete invariant of $W(\alpha)$ for the $B_1=0$ case (which is an abstract model for Chern--Simons theory on a rational homology sphere) and is given by
\begin{multline}
F(\hbar)=\hbar\log\left(e^{-\hbar\;\frac{1}{2}\pi^{-1}(\frac{\dd}{\dd\omega''},K\frac{\dd}{\dd\omega''})}|_{\omega''=0}\circ e^{\frac{1}{\hbar}\,\frac{1}{6}\pi(\omega'',l(\omega'',\omega''))}\right)=\\
=\sum_{l=2}^\infty \hbar^l \sum_{\Gamma^\mr{vac}\in G_{l,0}} \frac{1}{|\mr{Aut}(\Gamma^\mr{vac})|} F_{\Gamma^\mr{vac}} \label{F pert expansion}
\end{multline}
Where we sum over trivalent connected non-oriented graphs without leaves $\Gamma^\mr{vac}$ (the ``vacuum loops''). The contribution of a Feynman graph $F_{\Gamma^\mr{vac}}\in\RR$ is a number, computed by the same Feynman rules as for (\ref{W pert series}), just without the insertions of $\alpha$.

\begin{Exa}[Chevalley--Eilenberg complex of $\mathfrak{su}(2)$]
We obtain an interesting example of abstract Chern-Simons theory with $B_1=0$ if we choose
$$\C=S^\bt(\mathfrak{su}(2)^*[-1])=\RR\oplus \mathfrak{su}(2)^*[-1]\oplus \left(\wedge^2 \mathfrak{su}(2)^*\right)[-2]\oplus \left(\wedge^3 \mathfrak{su}(2)^*\right)[-3]$$
--- the Chevalley-Eilenberg complex of the Lie algebra $\su(2)$. This $\C$ is naturally a dg Frobenius algebra with super-commutative wedge product, Chevalley--Eilenberg differential
$$d:\quad e_1\mapsto e_2 e_3,\quad e_2\mapsto e_3 e_1,\quad e_3\mapsto e_1 e_2$$
and pairing
$$\pi({\bf{1}},e_1 e_2 e_3)=\pi(e_1, e_2 e_3)=\pi(e_2, e_3 e_1)=\pi(e_3, e_1 e_2)=1$$
Here $\{e_1,e_2,e_3\}$ is the basis in $\su(2)^*$, dual to the basis $\{-\frac{i}{2}\sigma_1,-\frac{i}{2}\sigma_2,-\frac{i}{2}\sigma_3\}$ in $\su(2)$, where $\{\sigma_i\}$ are the Pauli matrices; ${\bf{1}}$ is the unit in $\C^0=\RR$. This $\C$ can be understood as the algebra of left-invariant differential forms on the Lie group $SU(2)\sim S^3$ which is indeed quasi-isomorphic (as a dg algebra) to the whole de Rham algebra of the sphere $S^3$; thus the abstract Chern--Simons theory associated to this $\C$ is in a sense a toy model for true Chern--Simons theory on $S^3$. The
Hodge decomposition (\ref{Hodge decomp for C}) for $\C$ is unique:
$$\C=\underbrace{\RR\oplus \left(\wedge^3 \mathfrak{su}(2)^*\right)[-3]}_{H^\bt(\C)}\oplus \underbrace{\mathfrak{su}(2)^*[-1]}_{\C_{K-ex}}\oplus \underbrace{\left(\wedge^2 \mathfrak{su}(2)^*\right)[-2]}_{\C_{d-ex}} $$
and the BV field $\omega$ is
$$\omega=\underbrace{\alpha_{(0)}{\bf{1}}+\alpha_{(3)}e_1 e_2 e_3}_{\alpha} + \underbrace{\omega^1 e_1+\omega^2 e_2 + \omega^3 e_3}_{\omega''_{K-ex}}+ \underbrace{\omega^{23} e_2 e_3+\omega^{31} e_3 e_1+\omega^{12} e_1 e_2}_{\omega''_{d-ex}}$$
Here $\alpha_{(0)},\alpha_{(3)}$ are $\g$-valued coordinates on $\FF'=\g\otimes H^\bt(\C)[1]$ and $\{\omega^I\},\{\omega^{IJ}\}$ are $\g$-valued coordinates on $\FF''\subset g\otimes \C[1]$. The Lagrangian subspace (\ref{L_K}) is $\LL=\g\otimes\su(2)^*$.
The effective action on cohomology, as defined by the integral (\ref{W def by BV integral}), satisfies the ansatz (\ref{prop6 W}) with $F(\hbar)$ given by
\be F(\hbar)=\hbar\;\log\frac{\int_{\g\otimes \su(2)^*} e^{\frac{1}{\hbar}\left(\frac{1}{2}\sum_{I=1}^3 \pi_\g(\omega^I,\omega^I)+\pi_\g(\omega^1,[\omega^2,\omega^3]) \right)}d\omega^1 d\omega^2 d\omega^3}{\int_{\g\otimes \su(2)^*} e^{\frac{1}{\hbar}\cdot\frac{1}{2}\sum_{I=1}^3 \pi_\g(\omega^I,\omega^I)}d\omega^1 d\omega^2 d\omega^3} \label{F su(2) integral}\ee
The perturbative expansion (\ref{F pert expansion}) now reads
\begin{multline}
F(\hbar)=\hbar\log\left(\left.e^{-\hbar\;\frac{1}{2}\sum_{I=1}^3 \pi_\g^{-1}(\frac{\dd}{\dd\omega^I},\frac{\dd}{\dd\omega^I})}\right|_{\omega^I=0}\circ e^{\frac{1}{\hbar}\,\pi_\g(\omega^1,[\omega^2,\omega^3])}\right)=\\
=\sum_{l=2}^\infty (-1)^{l+1} \hbar^l \sum_{\Gamma^\mr{vac}\in G_{l,0}} \frac{1}{|\mr{Aut}(\Gamma^\mr{vac})|} L_{\Gamma^\mr{vac}}^{\su(2)}\cdot L_{\Gamma^\mr{vac}}^{\g}
\label{F su(2)}
\end{multline}
Here $L_{\Gamma^\mr{vac}}^\g$ denotes the ``Lie algebra graph'' (or the ``Jacobi graph''), i.e. the number\footnote{Strictly speaking, one also has to choose a cyclic ordering of half-edges for each vertex of $\Gamma^\mr{vac}$, and this choice affects the total sign of $L_{\Gamma^\mr{vac}}^\g$. But this ambiguity is cancelled in (\ref{F su(2)}) since the factors $L_{\Gamma^\mr{vac}}^\g$ and $L_{\Gamma^\mr{vac}}^{\su(2)}$ change their signs simultaneously if we change the cyclic ordering of half-edges in any vertex of $\Gamma^\mr{vac}$.} obtained by decorating vertices of $\Gamma^\mr{vac}$ with the structure constants $f_{abc}$ of the Lie algebra $\g$ and taking contraction of indices over edges of $\Gamma^\mr{vac}$. In particular, for $L_{\Gamma^\mr{vac}}^{\su(2)}$ vertices are decorated with structure costants of $\su(2)$ --- the Levi-Civita symbol\footnote{We assume that the inner product on $\su(2)$ is normalized as $\pi_{\su(2)}(x,y)=-2\, \tr(x y)$ (in the fundamental representation of $\su(2)$), so that the structure constants for the orthonormal basis are really $\epsilon_{IJK}$.} $\epsilon_{IJK}\in\{\pm 1,0\}$.
First terms in the series (\ref{F su(2)}) are:
$$
F(\hbar)
=-\hbar^2\,\frac{1}{12}\cdot 6\cdot f_{abc}f_{abc}+\hbar^3\left(\frac{1}{16}\cdot 12\cdot f_{acd}f_{bcd}f_{aef}f_{bef}+\frac{1}{24}\cdot 6\cdot f_{abc}f_{ade}f_{bef}f_{cfd}\right)+\cdots
$$
In particular, for $\g=\su(N)$ we have
$$F(\hbar)=-\hbar^2 \frac{1}{2}(N^3-N)+\hbar^3 \frac{7}{8} (N^4-N^2)-\hbar^4 \frac{23}{8} (N^5-N^3)+\cdots$$
As a side note, $L_{\Gamma^\mr{vac}}^{\su(2)}$ can be interpreted combinatorially as the number of ways to decorate edges of $\Gamma^\mr{vac}$ with 3 colors, such that in each vertex edges of all 3 colors meet (these decorations should be counted with signes, determined by the cyclic order of colors in each vertex). Thus for $\g=\su(2)$ the invariant $F(\hbar)$ is given by
$$F(\hbar)=\sum_{l=2}^\infty (-1)^{l+1} \hbar^l \sum_{\Gamma^\mr{vac}\in G_{l,0}} \frac{1}{|\mr{Aut}(\Gamma^\mr{vac})|} (L_{\Gamma^\mr{vac}}^{\su(2)})^2$$ and it can be viewed as a generating function for certain interesting combinatorial quantities (and on the other hand, there is an ``explicit'' integral formula (\ref{F su(2) integral}) for $F(\hbar)$ in terms of a 9-dimensional Airy-type integral).

\end{Exa}

\begin{proposition} \label{prop7}
Suppose $\C$ is formal and the embedding $\iota: H^\bt(\C)\hra \C$ is an algebra homomorphism, then:
\begin{itemize}
\item The tree part of the effective action on cohomology contains only the trivial $W_\mr{prod}$ term, i.e.
\be W(\alpha)=W_\mr{prod}(\alpha)+\hbar F^{(1)}(\alpha_{(1)}^1,\ldots,\alpha_{(1)}^{B_1})+\sum_{l=2}^\infty \hbar^l F^{(l)}(\alpha_{(1)}^1,\ldots,\alpha_{(1)}^{B_1}) \label{prop7 W}\ee
where $F^{(l)}\in (\Fun(\g^{B_1}))^\g$ for $l=1,2,\ldots$.
\item The 1-loop part of effective action can be written as
\be F^{(1)}(\alpha_{(1)})=\frac{1}{2}\;\Str_{\g\otimes\C}\log\left(1+K\circ l(\iota(\alpha_{(1)}),\bt)\right) \label{prop7 F^1 as Str}\ee
\item Restriction of the 1-loop part of the effective action $F^{(1)}|_\mr{MC}$ to the Maurer-Cartan set is invariant under deformations of $(\iota,K)$ (preserving the homomorphism condition for $\iota$). Here $\mr{MC}\subset H^1(\C,\g)[1]$ is given by
\[\sum_{j,k}\mu_{ijk}[\alpha_{(1)}^j,\alpha_{(1)}^k]=0.\]
\end{itemize}
\end{proposition}

\begin{proof}
The fact that $\iota$ is a homomorphism implies
\be K\circ l(\iota(\alpha),\iota(\alpha))=0 \label{prop7 eq1}\ee
This means that all Feynman diagrams for $W(\alpha)$ that contain a vertex adjacent to two leaves and one internal edge, vanish. In particular, all tree diagrams except $\Gamma_{0,3}$ vanish. Together with Proposition \ref{prop4} this implies (\ref{prop7 W}). Also (\ref{prop7 eq1}) implies that among 1-loop diagrams only ``wheels'' survive, and they are summed up to form (\ref{prop7 F^1 as Str}) in standard way.

The next observation is that the tree part $R'^0$ of the generator of the canonical transformation (\ref{R' pert series}) induced by changing $(\iota,K)$ vanishes. This is a consequence of properties (\ref{prop7 eq1}), $\delta\kappa\circ l(\iota(\alpha),\iota(\alpha))=0$ and $\pi(\delta I(\alpha),l(\iota(\alpha),\iota(\alpha)))=0$ (which all follow from the fact that $\iota$ is a homomorphism). In terms of the vector field $G$ (\ref{prop5 R'}), we have $G=\sum_{l=1}^\infty \hbar^l G^{(l)}$ and the $l$-loop part of effective action is transformed according to (\ref{prop5 G acts on F}) as
$$F^{(l)}\mapsto F^{(l)}+G^{(l)}\circ W_\mr{prod}^{111}+\sum_{i=1}^{l-1} G^{(l-i)}\circ F^{(i)}+\mr{div}\; G^{(l-1)}$$
In particular, the 1-loop part is transformed as
$$F^{(1)}\mapsto F^{(1)}+G^{(1)}\circ W_\mr{prod}^{111}$$
Since the Maurer-Cartan set is precisely the locus of stationary points of $W_\mr{prod}^{111}$, the restriction $F^{(1)}|_\mr{MC}$ is invariant. This finishes the proof.
\end{proof}

Special case of formal $\C$ is the case $B_1=1$. Here the Maurer-Cartan set is $\mr{MC}=H^1(\C,\g)[1]\cong \g$, so the 1-loop part of effective action $F^{(1)}\in \Fun(\g)^\g$ is invariant (without any restriction).

\subsection{Comments on relaxing the condition $K^2=0$ for chain homotopy}\label{s:relax}
Let us introduce the notation
\begin{multline}\label{Ind}
\mr{Ind}_{\iota,K}(\bar S)(\alpha):=\hbar\log\left(e^{-\hbar\;\frac{1}{2}\pi^{-1}(\frac{\dd}{\dd\omega''},K\frac{\dd}{\dd\omega''})}|_{\omega''=0}\circ e^{(\bar{S}(\iota(\alpha)+\omega'')-S_0(\omega''))/\hbar}\right)\\
\in \Fun(\FF')[[\hbar]]
\end{multline}
for the ``effective action'' for some (not necessarily abstract Chern--Simons) action $\bar{S}\in \Fun(\FF)[[\hbar]]$, defined by perturbation series, rather than by BV integral itself. Expression (\ref{Ind}) is indeed the perturbation series generated by the BV integral
\be e^{\mr{Ind}_{\iota,K}(\bar S)(\alpha)/\hbar}=\frac{1}{N}\int_{\LL_K}e^{\bar{S}(\iota(\alpha)+\omega'')/\hbar}\mu_{\LL_K} \label{Ind via BV integral}\ee
with normalization $N$ as before (\ref{N}). In particular for abstract Chern--Simons action $\bar{S}=S$ we recover the definition of $W$:\quad $\mr{Ind}_{\iota,K}(S)(\alpha)=W(\alpha)$.

The important observation with which we are concerned here is that definition (\ref{Ind}) makes sense for a chain homotopy $K$ not necessarily satisfying property $K^2=0$ (we assume that the other properties we demanded (\ref{dK+Kd},\ref{K self-adj},\ref{K iota=0}) hold), while the definition via BV integral (\ref{Ind via BV integral}) is less general and uses essentially $K^2=0$ property for construction of Lagrangian subspace $\LL_K$. To avoid confusion we will denote by $\hat{K}$ the chain homotopy without property (\ref{K^2=0}) and reserve notation $K$ for the ``honest'' chain homotopy with property (\ref{K^2=0}). We will call the effective action defined via (\ref{Ind}) with $\hat{K}$ as chain homotopy the ``relaxed'' effective action, while for an honest chain homotopy $K$ we call the effective action (defined equivalently by (\ref{Ind}) or by BV integral (\ref{Ind via BV integral})) ``strict''.

\begin{proposition} \label{prop8}
Let $\hat{K}:\C^\bt\ra\C^{\bt-1}$ be a chain homotopy satisfying properties (\ref{dK+Kd},\ref{K self-adj},\ref{K iota=0}), but with $\hat{K}^2\neq 0$, and let $K=\hat{K} d \hat{K}$ be the construction (\ref{K_3=KdK}) applied to $\hat{K}$ (i.e.\ $K$ satisfies all the properties (\ref{dK+Kd},\ref{K self-adj},\ref{K iota=0},\ref{K^2=0})). Then the relaxed effective action $\mr{Ind}_{\iota,\hat{K}}(S)$ is equivalent (i.e.\ connected by a canonical transformation) to the strict effective action $\mr{Ind}_{\iota,K}(S)$.
\end{proposition}

\begin{proof}
The first observation is that since $K=\hat{K} d \hat{K}$, we can write the relaxed chain homotopy as
\be \hat{K}=K+d \Lambda d \label{K+d Lambda d}\ee
where $\Lambda=\hat{K}^3: \C^3\ra \C^0$. In fact, formula (\ref{K+d Lambda d}), with arbitrary skew-symmetric, degree -3 linear map $\Lambda: \C^3\ra \C^0$, gives a general (finite) deformation of honest chain homotopy $K$, preserving properties (\ref{dK+Kd},\ref{K self-adj},\ref{K iota=0}), but violating (\ref{K^2=0}) and satisfying in addition $K=\hat{K} d \hat{K}$. In other words, deformation (\ref{K+d Lambda d}) is the inverse of projection (\ref{K_3=KdK}).

Second, we interpret the Feynman diagram decomposition for $\mr{Ind}_{\iota,\hat{K}}$ with propagator $\hat{K}$ given by (\ref{K+d Lambda d}) as sum over graphs with edges decorated either by $K$ or by $d\Lambda d$, and then raise the Feynman subgraphs with edges decorated only by $d\Lambda d$ into action. I.e.\ we obtain
\be \mr{Ind}_{\iota,\hat{K}}(S)=\mr{Ind}_{\iota,K}(S+\Phi_\Lambda) \label{prop8 eq1}\ee
where
\begin{multline}\Phi_\Lambda=\frac{1}{8}\pi(l(\omega,\omega), d\Lambda d\; l(\omega,\omega))+\frac{1}{8}\pi(l(\omega,\omega),d\Lambda d\; l(\omega,d\Lambda d\; l(\omega,\omega)))+\\
+\frac{1}{8}\pi(l(\omega,\omega),d\Lambda d\; l(\omega,d\Lambda d\; l(\omega,d\Lambda d\; l(\omega,\omega))))+\cdots+\\
+\hbar\; \frac{1}{2}\cdot\frac{1}{2} \Str_{\FF}\;d\Lambda d\;l(\omega,d\Lambda d\;l(\omega,\bt))+\hbar\;\frac{1}{2}\cdot\frac{1}{3} \Str_{\FF}\;d\Lambda d\;l(\omega,d\Lambda d\;l(\omega,d\Lambda d\;l(\omega,\bt)))+\cdots
\label{Phi_Lambda}
\end{multline}
Only the simplest trees (``branches'') and only the simplest one-loop diagrams (``wheels'') contribute to $\Phi_\Lambda$, because any diagram with 3 incident internal edges decorated by $d\Lambda d$ automatically vanishes due to
$$\pi(d\Lambda d(\cdots),l(d\Lambda d(\cdots),d\Lambda d(\cdots)))=0$$
(which is implied by Leibniz identity in $\g\otimes \C$, by $d^2=0$ and skew-symmetry of $d$).

Third, we notice that there is a canonical transformation from $S$ to $S+\Phi_\Lambda$. A convenient way to describe a finite canonical transformation is to present a ``homotopy''
$$S_\Lambda(t,dt)=S_{\Lambda}(t)+dt\cdot R_\Lambda(t) \in \Fun(\FF)[[\hbar]]\otimes \Omega^\bt ([0,1])$$ --- a differential form on interval $[0,1]$ with values in functions on $\FF$ ($t\in[0,1]$ is a coordinate on the interval), satisfying the QME on the extended space $\FF\oplus T[1]([0,1])$:
\be (d_t+\hbar\Delta) S_\Lambda(t,dt)+\frac{1}{2}\{S_\Lambda(t,dt),S_\Lambda(t,dt)\}=0 \label{extended QME}\ee
(where $d_t=dt \frac{\dd}{\dd t}$ is the de Rham differential on the interval), and satisfying boundary conditions
$$S_\Lambda(t)|_{t=0}=S,\quad S_\Lambda(t)|_{t=1}=S+\Phi_\Lambda$$
The extended QME is equivalent to the fact that $S_\Lambda(t)$ is a solution to the QME on $\FF$ for any given $t\in[0,1]$ plus the fact that $S_\Lambda(t+\delta t)$ is obtained from $S_\Lambda(t)$ by the infinitesimal canonical transformation with generator $\delta t\cdot R_\Lambda(t)$.
In our case it is a straightforward exercise
to present the desired homotopy between $S$ and $S+\Phi_\Lambda$:
$$S_\Lambda(t,dt)=S+\Phi_{t\Lambda}+dt \frac{1}{t}R_{t\Lambda}$$
where $\Phi_{t\Lambda}$ is defined by (\ref{Phi_Lambda}) with $\Lambda$ rescaled by $t$, and $R_\Lambda$ is given by
\begin{multline*}
R_\Lambda=\frac{1}{4}\pi(d\Lambda d\;\omega,l(\omega,\omega))+\frac{1}{4} \pi(d\Lambda d\;\omega,l(\omega,d\Lambda d\;l(\omega,\omega)))+\\
+\frac{1}{4} \pi(d\Lambda d\;\omega,l(\omega,d\Lambda d\;l(\omega,d\Lambda d\;l(\omega,\omega))))+\cdots
\end{multline*}

We showed that $S$ and $S+\Phi_\Lambda$ are connected by a homotopy. Due to Proposition \ref{prop1} and Lemma \ref{lm1} the effective actions $\mr{Ind}_{\iota,K}(S)$ and $\mr{Ind}_{\iota,K}(S+\Phi_\Lambda)$ are also connected by a homotopy defined by \be e^{W_{\Lambda}(t,dt)(\alpha)/\hbar}=\frac{1}{N}\int_{\LL_K}e^{S_\Lambda(t,dt)(\iota(\alpha)+\omega'')/\hbar}\mu_{\LL_K} \label{W_Lambda via BV integral}\ee
Together with (\ref{prop8 eq1}) this finishes the proof.
\end{proof}

\begin{Rem} An immediate consequence of Proposition \ref{prop8} is that the relaxed effective action $\mr{Ind}_{\iota,\hat{K}}(S)$ satisfies the QME on $\FF'$.
\end{Rem}

\begin{Rem} Expression (\ref{Phi_Lambda}) suggests that it is itself the value of a certain Gaussian integral. Namely, the restriction of $S+\Phi_\Lambda$ to the subspace $\iota(\FF')\oplus\LL_K\subset \FF$ can be written as a fiber integral over fibers $\g\otimes \C_{d-ex}^1[1]$ of vector bundle $\iota(\FF')\oplus\LL_K\oplus \g\otimes \C_{d-ex}^1[1]\ra \iota(\FF')\oplus\LL_K$ (which is a sub-bundle of $\FF\ra \iota(\FF')\oplus\LL_K$):
$$\left. e^{(S+\Phi_\Lambda)/\hbar}\right|_{\iota(\FF')\oplus\LL_K}=\int_{\g\otimes \C_{d-ex}^1[1]}e^{(S+\frac{1}{2}\pi(\omega,K\Lambda^{-1}K\omega))/\hbar}\mu_{\g\otimes \C_{d-ex}^1[1]}$$
(by $\mu_{(\cdots)}$ we always mean the Lebesgue measure on the vector space). Here we assume for simplicity that $\Lambda: \C^3\ra \C^0$ is an isomorphism, and we denote $\Lambda^{-1}:\C^0\ra \C^3$ its inverse. Now we can write the relaxed effective action as
\be e^{\mr{Ind}_{\iota,\hat{K}}(S)(\alpha)/\hbar}=\int_{\LL_K\oplus \g\otimes \C_{d-ex}^1[1]} e^{S(\iota(\alpha)+\omega'')/\hbar}\;e^{\frac{1}{2\hbar}\pi(\omega'',K\Lambda^{-1}K\omega'')} \mu_{\LL_K\oplus \g\otimes \C_{d-ex}^1[1]} \label{thick BV integral}\ee
where we integrate over the coisotropic subspace $\LL_K\oplus \g\otimes \C_{d-ex}^1[1]\subset\FF''$ instead of just the Lagrangian subspace $\LL_K\subset\FF''$, with measure $e^{\frac{1}{2\hbar}\pi(\omega,K\Lambda^{-1}K\omega)} \mu_{\LL_K\oplus \g\otimes \C_{d-ex}^1[1]}$ that is constant in direction of $\LL_K$ and is Gaussian in direction of $\g\otimes \C_{d-ex}^1[1]$. So expression (\ref{thick BV integral}) gives an elegant interpretation of the relaxed effective action via a ``thick'' fiber BV integral (over a Gaussian-smeared Lagrangian subspace). In the case of $\Lambda$ of general rank (not necessarily an isomorphism), we should replace $C_{d-ex}^1$ by $\mr{im}(d\Lambda)\subset C_{d-ex}^1$ in this discussion, which leads to a thick fiber BV integral over a smaller coisotropic $\LL_K\oplus \g\otimes \mr{im}(d\Lambda)[1]\subset\FF''$.
\end{Rem}

\subsubsection{Invariants from the relaxed effective action} \label{invreleffact}
We are interested in describing the invariants of the relaxed effective action $\Ind_{\iota,\hat{K}}(S)$ modulo deformations of the ``relaxed induction data'' $(\iota,\hat{K})$. Since we have a general description (\ref{K+d Lambda d}) for a non-strict chain homotopy, a general (infinitesimal) deformation of the relaxed induction data $(\iota,\hat{K})$ can be described as a deformation of the underlying strict induction data $(\iota,K)$ studied in the beginning of Section~\ref{sec 2.3}, plus a deformation of $\Lambda$. Now we can restate some weak version of Proposition \ref{prop6} for the case of relaxed induction, where we make a special choice for the Lie algebra of coefficients: $\g=\su(2)$ (we generalize this in the Remark afterwards). We are able to recover only the two-loop part of the complete invariant $F(\hbar)$ of the strict effective action in this discussion.

\begin{proposition} \label{prop9}
If $B_1=0$ and $\g=\su(2)$, the relaxed effective action on cohomology has the form
\be \Ind_{\iota,\hat{K}}(S)(\alpha)=A(\hbar)\cdot W_\mr{prod}^{003}(\alpha)+B(\hbar) \label{A W_prod + B}\ee
where  $W_\mr{prod}^{003}=\frac{1}{2}\sum_{a,b,c}\epsilon_{abc}\alpha_{(0)}^a \alpha_{(0)}^b \alpha_{(3)}^c$, and $A(\hbar)= 1+\hbar A^{(1)}+\hbar^2 A^{(2)}+\cdots$,\quad $B(\hbar)=\hbar^2 B^{(2)}+ \hbar^3 B^{(3)}+\cdots$ are some $\hbar$-dependent constants. The number
\be F^{(2)}=3 A^{(1)}+B^{(2)} \label{prop9 3A^1+B^2}\ee
is invariant under deformations of the relaxed induction data $(\iota,\hat{K})$.
\end{proposition}

\begin{proof}
Ansatz (\ref{A W_prod + B}) follows from the following argument. Since the relaxed effective action $\hat{W}=\Ind_{\iota,\hat{K}}(S)\in \Fun\left(\g[1]\oplus \g[-2]\right)[[\hbar]]$ is a function of ghost number zero and $\g=\su(2)$ is 3-dimensional (and hence an at most cubic dependence on $\alpha_{(0)}$ is possible), it has to be of
the form
$$\hat{W}(\alpha)=\frac{1}{2}\sum_{a,b,c}A_{abc}(\hbar)\alpha_{(0)}^a \alpha_{(0)}^b \alpha_{(3)}^c+B(\hbar)$$
Since $\hat{W}$ also inherits the invariance under adjoint action of $\G=SU(2)$:\quad $\alpha\mapsto g\alpha g^{-1}$ for $g\in\G$ from the ad-invariance of the original abstract Chern--Simons action $S$, the tensor $A_{abc}(\hbar)$ has to be of the form $A_{abc}(\hbar)=A(\hbar)\epsilon_{abc}$. This proves (\ref{A W_prod + B}). The fact that the series for $A(\hbar)$ starts as $A(\hbar)=1+O(\hbar)$ is due to the vanishing of all tree diagrams $\Gamma\neq \Gamma_{0,3}$ which follows from formality of $\C$ (it is automatic for the $B_1=0$ case) and $\hat{K}\circ\iota=0$. The series for $B(\hbar)$ starts with an $O(\hbar^2)$-term just because $G_{l,0}$ is empty for $n=0,1$.

Next, we know from Proposition \ref{prop8} that there is a homotopy $W_\Lambda(t,dt)(\alpha)$ connecting the strict effective action $W=\Ind_{\iota,K}(S)=W^{003}_\mr{prod}(\alpha)+F(\hbar)$ (ansatz (\ref{prop6 W})) and the relaxed one $\hat{W}=\Ind_{\iota,\hat{K}}(S)$. The general ansatz for $W_\Lambda(t,dt)$, taking into account that $\g=\su(2)$ and that construction (\ref{W_Lambda via BV integral}) is compatible with ad-invariance $\alpha\mapsto g\alpha g^{-1}$, is the following:
\begin{multline} \label{prop9 W_Lambda ansatz}
W_\Lambda(t,dt)(\alpha)=A(\hbar;t) W_\mr{prod}^{003}(\alpha)+ B(\hbar;t)+\\
+dt\cdot\left( C(\hbar;t)\; \frac{1}{12} \sum_{a,b,c,d,e} \epsilon_{abc}\delta_{de}\alpha_{(0)}^a \alpha_{(0)}^b \alpha_{(0)}^c \alpha_{(3)}^d \alpha_{(3)}^e + D(\hbar;t) \sum_{a,b} \delta_{ab} \alpha_{(0)}^a \alpha_{(3)}^b \right)
\end{multline}
where $A,B,C,D$ are some functions of $\hbar$ and the homotopy parameter $t\in [0,1]$. The boundary conditions are:
$$A(\hbar;0)=1,\; A(\hbar;1)=A(\hbar),\;B(\hbar;0)=F(\hbar),\; B(\hbar;1)=B(\hbar)$$
The extended QME (\ref{extended QME}) for homotopy $W_\Lambda(t,dt)$ is equivalent to the system
\begin{eqnarray}
\frac{\dd}{\dd t} A(\hbar;t)=\hbar\; C(\hbar;t)- A(\hbar;t)\cdot D(\hbar;t) \label{prop9 eq1}\\
\frac{\dd}{\dd t} B(\hbar;t)=3 \hbar\; D(\hbar;t) \label{prop9 eq2}
\end{eqnarray}

Next, the argument of vanishing of non-trivial trees (due to formality of $\C$,\; $K\circ\iota=0$ and $d\circ \iota=0$) applies again to construction (\ref{W_Lambda via BV integral}). Hence, we have
$$A(\hbar;t)=1+\hbar A^{(1)}(t)+O(\hbar^2), \; C(\hbar;t)=\hbar C^{(1)}(t)+O(\hbar^2),\; D(\hbar;t)=\hbar D^{(1)}(t)+O(\hbar^2)$$
And due to $G_{0,0}=G_{1,0}=\varnothing$, we again have $B(\hbar;t)=\hbar^2 B^{(2)}(t)+O(\hbar^3)$. Equation (\ref{prop9 eq1}) in order $O(\hbar)$ together with (\ref{prop9 eq2}) in order $O(\hbar^2)$ yield
$$\frac{\dd}{\dd t}A^{(1)}(t)=-D^{(1)}(t),\quad \frac{\dd}{\dd t} B^{(2)}(t)=3 D^{(1)}(t)$$
Hence the expression $3 A^{(1)}(t)+B^{(2)}(t)$ does not depend on $t$. For $t=0$ it is the two-loop part of the invariant $F(\hbar)$ from (\ref{prop6 W}), while for $t=1$ it is the right hand side of (\ref{prop9 3A^1+B^2}). This concludes the proof.
\end{proof}

\begin{Rem}[Generalization] The generalization of Proposition \ref{prop9} to an arbitrary (quadratic) $\g$ is straightforward. We no longer have ansatz (\ref{A W_prod + B}) for $\hat{W}$, since there might be more ad-invariant functions on $\g[1]\oplus \g[-2]$, but we still can write
$$\hat{W}(\alpha)=(1+\hbar A^{(1)}) W_\mr{prod}^{003}(\alpha)+\hbar^2 B^{(2)}+ O(\hbar^3+ \hbar^2 (\alpha_{(0)})^2\alpha_{(3)}+\hbar (\alpha_{(0)})^4 (\alpha_{(3)})^2)$$
(we could prescribe weight $1$ to $\hbar$ and weight $1/3$ to $\alpha_{(0)}$ and $\alpha_{(3)}$, then we write explicitly terms with weight $\leq 2$). The fact that $O(\hbar \alpha_{(0)} \alpha_{(0)} \alpha_{(3)})$-contribution is proportional to $W_\mr{prod}^{003}$ (in principle there could be some other invariant tensor of rank 3) is explained by the fact that it is given by a single Feynman diagram $\Gamma\in G_{1,3}$ --- the wheel with 3 leaves --- and the Lie algebra part of this diagram is described by contraction $\sum_{d,e,f}f_{ade}f_{bef}f_{cfd}\propto f_{abc}$. Following the proof of Proposition \ref{prop9}, for the homotopy we have
\begin{multline*} W_\Lambda(t,dt)(\alpha)=\\
(1+\hbar A^{(1)}(t))\; W_\mr{prod}^{003}(\alpha)+\hbar^2 B^{(2)}(t)+ O(\hbar^3+ \hbar^2 (\alpha_{(0)})^2 \alpha_{(3)}+\hbar (\alpha_{(0)})^4 (\alpha_{(3)})^2)+\\
+dt\cdot\left(\hbar\; D^{(1)}(t)\sum_{a}\alpha_{(0)}^a\alpha_{(3)}^a+O\left(\hbar^2 \alpha_{(0)}\alpha_{(3)}+\hbar (\alpha_{(0)})^3(\alpha_{(3)})^2 \right)\right)
\end{multline*}
The extended QME at order $O(dt\;\hbar\; (\alpha_{(0)})^2\alpha_{(3)}+dt\;\hbar^2)$ yields
$$\frac{\dd}{\dd t}A^{(1)}(t)=-D^{(1)}(t),\quad \frac{\dd}{\dd t} B^{(2)}(t)=\dim\g\cdot D^{(1)}(t)$$
Hence the two-loop part of the invariant $F(\hbar)$ is expressed in terms of coefficients of the relaxed effective action $\hat{W}$ as
\be F^{(2)}=\dim\g\cdot A^{(1)}+B^{(2)} \label{F^2 invariant}\ee
\end{Rem}

The other case discussed in Section~\ref{sec 2.4}, the case of formal $\C$ with general $B_1$, is translated straightforwardly into the setting of relaxed effective actions.
\begin{proposition} \label{prop10}
Suppose $\C$ is formal and $\iota: H^\bt(\C)\hra\C$ is an algebra homomorphism. Then the relaxed effective action has the form
\be \Ind_{\iota,\hat{K}}(S)(\alpha)=W_\mr{prod}(\alpha)+\hbar \hat{F}^{(1)}(\alpha)+O(\hbar^2) \label{prop10 eq1}\ee
where the one-loop part can be expressed as a super-trace:
\be \hat{F}^{(1)}(\alpha)=\frac{1}{2}\;\Str_{\g\otimes\C}\log \left(1+\hat{K}\circ l(\iota(\alpha),\bt)\right) \label{prop10 Str}\ee
The restriction of the one-loop relaxed effective action to the Maurer-Cartan set $\hat{F}^{(1)}|_\mr{MC}$ is invariant under deformations of $(\iota,\hat{K})$, preserving the homomorphism property of $\iota$.
\end{proposition}

\begin{proof}
Analogously to the case of a strict chain homotopy $K$ (Proposition \ref{prop7}), formality of $\C$ together with $\hat{K}\circ\iota=0$ imply vanishing of non-trivial trees (hence the ansatz (\ref{prop10 eq1})) and that the only possibly non-vanishing one-loop diagrams are wheels (hence the super-trace formula (\ref{prop10 Str})).

Second, we know that there is a homotopy $W_{\Lambda}(t,dt)(\alpha)$ connecting the strict effective action $W(\alpha)=W_\mr{prod}(\alpha)+\hbar F^{(1)}(\alpha_{(1)})+\cdots$ to the relaxed one $\hat{W}(\alpha)=W_\mr{prod}(\alpha)+\hbar \hat{F}^{(1)}(\alpha)+\cdots$:
$$W_{\Lambda}(t,dt)(\alpha)=\left(W_\mr{prod}(\alpha)+\hbar\; F^{(1)}(t)(\alpha)+O(\hbar^2)\right)+dt\cdot\left(\hbar\; R'^{(1)}_{\Lambda}(t)(\alpha)+O(\hbar^2)\right)$$
(here we again exploit the vanishing of trees implied by construction (\ref{W_Lambda via BV integral}), formality of $\C$ and $K\circ\iota=d\circ\iota= 0$). The extended QME for the homotopy at order $O(dt\cdot \hbar)$ is
$$\frac{\dd}{\dd t}F^{(1)}(t)(\alpha)=\{W_\mr{prod}(\alpha),R'^{(1)}_{\Lambda}(t)(\alpha)\}'$$
Hence the restriction $F^{(1)}(t)|_{\overline{\mr{MC}}}$ does not depend on $t$, where
\[
\overline{\mr{MC}}=\{\alpha|\;l(\iota(\alpha),\iota(\alpha))=0\}\subset\FF'
\]
is the set of critical points of $W_\mr{prod}$ (the ``non-homogeneous Maurer-Cartan set''). As $\mr{MC}\subset\overline{\mr{MC}}$, the restriction $F^{(1)}(t)|_{\mr{MC}}$ also does not depend on $t$, hence $$\hat{F}^{(1)}|_{\mr{MC}}=F^{(1)}|_{\mr{MC}}$$
As the right hand side is invariant (Proposition \ref{prop7}), so is the left hand side. This concludes the proof.
\end{proof}

Obviously, going from the restriction to $\overline{\mr{MC}}$ to the restriction to $\mr{MC}$, we do not lose any invariant information, since $F^{(1)}$ depends only on $\alpha_{(1)}$ and not on other components of $\alpha$, which implies $\hat{F}^{(1)}|_{\overline{\mr{MC}}}=\hat{F}^{(1)}|_\mr{MC}$. Note also that in Proposition \ref{prop9} we managed to reconstruct  only the two-loop part of the invariant $F(\hbar)$ (Proposition \ref{prop6}) in the relaxed setting. On the other hand, we can completely reconstruct the invariant $F^{(1)}|_\mr{MC}$ of Proposition \ref{prop7} in the relaxed setting.

\subsection{Examples of dg Frobenius algebras}
In this Section we will provide some examples of non-negatively graded dg Frobenius algebras with pairing of degree $-3$ and zeroth Betti number $B_0=1$, i.e.\ algebras suitable for constructing abstract Chern--Simons actions.

\textbf{Example 1: minimal dg Frobenius algebra.} Let $V$ be a vector space and $\mu\in \wedge^3 V^*$ an arbitrary exterior 3-form on $V$. Then we construct the dg Frobenius algebra $\C$ from this data as
$$\C:=\RR\cdot 1 \oplus V[-1]\oplus V^*[-2]\oplus \RR\cdot v$$
where $v$ is a degree 3 element. The pairing $\pi$ is defined to be the canonical pairing between $V[-1]$ and $V^*[-2]$, and also we set $\pi(1,v):=1$. We define the product $m$ as
$$m(1,x)=x,\quad m(x_{(1)},y_{(1)})=<\mu,\bar x_{(1)}\wedge \bar y_{(1)}>,\quad m(x_{(1)},z_{(2)})=<\bar x_{(1)},\bar z_{(2)}>$$
where $x\in\C$ (element of arbitrary degree),\; $x_{(1)},y_{(1)}\in V[-1]$,\; $z_{(2)}\in V^*[-2]$ and bar means shifting an element to degree zero, e.g.\ $\bar x_{(1)}= s x_{(1)}\in V$, $\bar z_{(2)}=s^2 z_{(2)}\in V^*$;\; $<\bt,\bt>$ denotes the canonical pairing. The differential $d$ is set to zero. This construction gives the most general minimal (i.e.\ with zero differential) dg Frobenius algebra (non-negatively graded, with pairing of degree -3 and $B_0=1$). Abstract Chern--Simons action associated to a minimal algebra is a purely cubic polynomial in the fields, with coefficients being the components of $\mu$. Inducing the effective action on cohomology is the identity operation, since here $\C=H^\bt(\C)$.

\textbf{Example 2: differential concentrated in degree $\C^1\ra \C^2$.} Let $V$ be a vector space, $\mu\in\wedge^3 V^*$ be some 3-form on $V$ and $\delta: S^2 V\ra \RR$ some symmetric pairing on $V$ (not necessarily non-degenerate). We define $\C$, the pairing $\pi$ and product $m$ as in Example 1, but now we construct the differential $d:V[-1]\ra V^*[-2]$ from $\delta$ as $x_{(1)}\mapsto s^{-2}\delta(\bar x_{(1)},\bt)$. The two other components of the differential $\RR\cdot 1\ra V[-1]$ and $V^*[-2]\ra \RR\cdot v$ are set to zero as before. This construction gives the general dg Frobenius algebra with differential concentrated in degree $\C^1[-1]\ra \C^2[-2]$ only.

Here the Hodge decomposition $\C=\iota(H^\bt(\C))\oplus \C_{d-ex}\oplus \C_{K-ex}$ is defined by a choice of projection $p: V\ra \ker\delta$ or equivalently by choosing a complement $V''$ of $\ker \delta\subset V$ in $V$ (here we understand $\delta$ as the self-dual map $V\ra V^*$). We set
\begin{multline*}
H^\bt(\C)=\RR\cdot 1\oplus (\ker\delta) [-1]\oplus (\ker\delta)^*[-2]\oplus \RR\cdot v,\\
\C_{d-ex}=(\im\delta)[-2],\quad \C_{K-ex}=V''[-1]
\end{multline*}
The embedding $\iota: H^\bt(\C)\hra \C$ is canonical in degrees 0,1,3 and given by $p^*: (\ker\delta)^*\hra V^*$ in degree 2. The (non-vanishing part of the) chain homotopy $K$ is the inverse map for the isomorphism $d: V''[-1]\ra (\im\delta)[-2]$. So the induction data $(\iota,K)$ is completely determined by the choice of $p$. This means in particular that only the deformations of induction data of type II are possible here. Also there are no relaxed chain homotopies $\hat{K}$ (other than the strict one
described above), which is obvious from (\ref{K+d Lambda d}). Despite these simplifications, the effective action on cohomology for such $\C$ is in general non-trivial. A particular example here is the case $\C=S^\bt(\su(2)^*[-1])$ discussed in section \ref{sec 2.4}.

\textbf{Example 3: ``doubled'' commutative dga.} Let $\V=V^0\oplus V^1[-1]$ be a unital commutative associative dg algebra, concentrated in degrees 0 and 1, with differential $d_\V$ and multiplication $m_\V$, and satisfying $\dim H^0(\V)=1$. Then we set
$$\C:=\V\oplus \V^*[-3]=V^0\oplus V^1[-1]\oplus (V^1)^*[-2]\oplus (V^0)^*[-3]$$
with pairing $\pi$ generated by the canonical pairing between $\V$ and $\V^*$. The component of differential $V^0\ra V^1[-1]$ is given by $d_\V$, the component $(V^1)^*[-2]\ra (V^0)^*[-3]$ --- by the dual map $(d_\V)^*$, and the component $V^1[-1]\ra (V^1)^*[-2]$ is set to zero. Multiplication for elements of $\C$ of degrees 0 and 1 is given by $m_\V$ and is extended to other degrees by the cyclicity property (\ref{cyclicity of m}):
$$ m(x_{(0,1)},y_{(0,1)}):=m_\V(x_{(0,1)},y_{(0,1)}),\quad m(x_{(0,1)},z_{(2,3)}):=<m_\V(\bt,x_{(0,1)}),z_{(2,3)}>$$
for $x_{(0,1)},y_{(0,1)}\in\V$ and $z_{(2,3)}\in\V^*[-3]$. In particular, the product of elements of degree 1 in $\C$ is zero.

Hodge decomposition for $\C$ is fixed by choosing an embedding $\iota_\V: H^\bt(\V)\hra \V$ (it is canonical in degree zero since $H^0(\V)=\RR\cdot 1$, but non-canonical in degree one) and a retraction
$r_\V: \V\ra H^\bt(V)$. Equivalently, we choose a splitting of $\V$ into representatives of cohomology and the complement $\V=\iota_\V(H^\bt(\V))\oplus \V''$, i.e.\ $V^0=\RR\cdot 1\oplus {V^0}''$, $V^1=\iota_\V(H^1(\V))\oplus {V^1}''$. Then the Hodge decomposition for $\C$ is
$\C=\iota(H^\bt(\C))\oplus \C_{d-ex}\oplus \C_{K-ex}$ with
\begin{eqnarray*}H^\bt(\C)=\RR\oplus H^1(\V)[-1]\oplus (H^1(\V))^*[-2]\oplus \RR\cdot v,\\
\C_{d-ex}={V^1}''[-1] \oplus ({V^0}'')^*[-3],\quad \C_{K-ex}={V^0}''\oplus ({V^1}'')^*[-2]
\end{eqnarray*}
Here $v\in (V^0)^*[-3]$ is the element defined by the component of retraction $r_\V: V^0\ra \RR\cdot 1$. The embedding $\iota: H^\bt(\C)\ra \C$ is given by $\iota_\V$ in degrees 0,1 and by $(r_\V)^*$ in degrees 2,3. As in the Example 2, only type II deformations of the induction data $(\iota,K)$ are possible here. However, one can introduce the relaxed chain homotopy (\ref{K+d Lambda d}) here with $\Lambda: ({V^0}'')^*[-3]\ra {V^0}''$.

The BV integral (\ref{W def by BV integral}) is Gaussian in this example since the cyclic product $\pi(\bt,m(\bt,\bt))$ vanishes on $(\C_{K-ex})^{\otimes 3}$ (for trivial degree reasons). Hence, the only Feynman graphs contributing to $W(\alpha)$ are wheels and the trivial tree $\Gamma_{0,3}$ (``branches'' could also contribute, but they vanish since the component of the multiplication $\iota(H^1(\C))\otimes \iota(H^1(\C))\ra \C^2$ vanishes and the non-trivial part of $W(\alpha)$ depends only on $\alpha_{(1)}$ due to Proposition \ref{prop4}). The effective action can be written as
$$W(\alpha)=W^{003}_\mr{prod}(\alpha)+W^{012}_\mr{prod}(\alpha)+2\cdot\frac{1}{2}\;\hbar\; \tr_{\g\otimes V^0}\log\left(1+K\circ l(\iota(\alpha_{(1)}),\bt)\right)$$
(the factor 2  accounts for the contribution of the trace over $\g\otimes ({V^1}'')^*$ --- the other half of the Lagrangian subspace $\LL_K$).

\section{Three-manifold invariants}
We now wish to extend the results of the previous Sections to the Frobenius algebra $\Omega^\bullet(M)$ of differential forms on a
smooth compact $3$\ndash manifold $M$ (with de~Rham differential, wedge product and integration pairing). This will provide us
with the effective action (around the trivial connection) of Chern--Simons theory \cite{W}. The invariant of this Frobenius algebra will
also constitute an invariant of $3$\ndash manifolds modulo diffeomorphisms.

The discussion of the previous Sections, however, does not go through automatically since $\Omega^\bullet(M)$ is infinite dimensional.
As in previous works \cite{AS,K,BC} the way out is to restrict oneself to a special class of chain homotopies $K$ for which
the finite dimensional arguments are simply replaced by the application of Stokes' theorem.
However, at some point of the construction we have to choose a
framing and
hence we get invariants of framed $3$\ndash manifolds.

\subsection{Induction data and the propagator}
As in the finite-dimensional case we fix an embedding $\iota\colon H^\bullet(M)\hra\Omega^\bullet(M)$. By $\chi\in\Omega^3(M\times M)$ we will denote the
representative of the Poincar\'e dual of the diagonal $\Delta$ determined by this embedding. Namely, if $\left\{1,\{\alpha_i,\beta^i\}_{i=1,\dots,B_1},v\right\}$
is a basis of $\iota(H^\bullet(M))$ (with $1$ the constant function, $\alpha_i\in\Omega^1$, $\beta^i\in\Omega^2$, $v\in\Omega^3$ and
$\int_M \alpha_i\,\beta^j=\delta_i^j$, $\int_Mv=1$), then
\[
\chi = v_2 - \alpha_{i,1}\beta^i_2 + \beta^i_1\alpha_{i,2} - v_1,
\]
where we have used Einstein's convention over repeated indices and for any form $\gamma\in\Omega^\bullet(M)$ we write $\gamma_i$ for
$\pi_i^*\gamma$ with $\pi_1$ and $\pi_2$ the two projections $M\times M\to M$.

The chain homotopy $K$ is assumed to be determined by a smooth integral kernel $\hat\eta$. Namely, let
$C_2^0(M):=M\times M\setminus\Delta=\{(x_1,x_2)\in M\times M : x_1\not=x_2\}$ be the open configuration space of two points in $M$ and let
$C_2(M)$ be its Fulton--MacPherson--Axelrod--Singer (FMAS) compactification \cite{FM,AS} obtained by replacing the diagonal $\Delta$ with
its unit normal bundle. Let $\pi_{1,2}$ be the extensions to the compactification of the projection maps $\pi_i(x_1,x_2)=x_i$.
Then $\hat\eta$ is a smooth $2$\ndash form on $C_2(M)$ and the chain homotopy $K$ is defined by
\[
K\alpha = -\pi_{2*}(\hat\eta\,\pi_1^*\alpha),\qquad\alpha\in\Omega^\bullet(M),
\]
where a lower $*$ denotes integration along the fiber.

For $K$ to be a symmetric chain homotopy satisfying \eqref{dK+Kd}, $\hat\eta$ must satisfy the following properties (see \cite{BC,C} for more details):
\begin{description}
\item[P1] $\ddd\hat\eta=\pi^*\chi$, with $\pi$ the extension to $C_2(M)$ of the inclusion of $C_2^0(M)$ into $M\times M$.
\item[P2] $\int_{\de C_2(M)}\hat\eta=-1$.
\item[P3] $T^*\hat\eta=-\hat\eta$, where $T$ is the extension to $C_2(M)$ of the involution $(x_1,x_2)\mapsto(x_2,x_1)$ of $C_2^0(M)$.
\end{description}

Observe that $\de C_2(M)$ is canonically diffeomorphic to the sphere tangent bundle $STM$ of $M$ ($STM$ is the quotient of $TM$ by the action
$(x,v)\mapsto(x,\lambda v)$, $\lambda\in\mathbb{R}^+_*$). Condition \textsf{P2} may then be refined to
\begin{description}
\item[P2'] $\iota_\de^*\hat\eta=-\eta$, where $\iota_\de$ is the inclusion map $STM=\de C_2(M)\hra C_2(M)$ and
$\eta$ is a given odd global angular form on $STM$.
\end{description}
Recall that a global angular form on a sphere bundle $S\to M$ is a differential form on the total space $S$ whose restriction
to each fiber generates its cohomology and whose differential is minus the pullback of a representative of the Euler class (in our case zero).
These two properties are consistent with the restriction to the boundary of \textsf{P1} and \textsf{P2}. Odd here means
$T^*\eta=-\eta$ where $T$ is the antipodal map on each fiber, and this is compatible with the restrictions to the boundary of \textsf{P3}.
Global angular forms always exist.

\begin{Rem}[Kontsevich \cite{K}]\label{r:K}
Since $STM$ is trivial for every $3$\ndash manifold, there is a simple choice for $\eta$
only depending on a choice of framing or, equivalently, a choice of trivialization $f\colon STM \stackrel\sim\to M\times S^2$.
One simply sets $\eta=f^*\omega$ where  $\omega$ is the normalized $SO(3)$\ndash invariant volume form on $S^2$ (tensor $1\in\Omega^0(M)$).
\end{Rem}
\begin{Rem}[\cite{BC}]\label{r:BC}
It is also possible to construct an odd global angular form depending on the choice of a connection but not on a framing. 
Namely, realize $STM$ as the $S^2$\ndash bundle associated to the frame bundle $F(M)$:
\[
STM=F(M)\times_{SO(3)}S^2
\]
(we reduce the structure group
of the frame bundle to $SO(3)$ by picking a Riemannian metric).
By choosing a metric connection $\theta$ (e.g., the Levi-Civita connection),
one defines $\bar\eta:=\omega+\ddd(\theta_i x^i)/(2\pi)$; here the $x^i$s, $i=1,2,3$, are the homogeneous coordinates on $S^2$, while
the $\theta_i$s are the coefficients
of the connection in the basis $\{\xi_i\}$ of $\mathfrak{so(3)}$ given by $(\xi_i)_{jk}=\epsilon_{ijk}$.
It is then easy to show that $\bar\eta$ is a global angular form on $F(M)\times S^2$ and that it is basic. Hence it defines a global angular
form $\eta$ on $STM$.
\end{Rem}

\begin{Lem}[\cite{BC}]\label{l:P123}
For any given odd global angular form $\eta$ on $STM$, there exists a propagator $\hat\eta\in\Omega^2(C_2(M))$ satisfying
\textsf{P1, P2', P3}.
\end{Lem}

\begin{proof}
The complete proof in the case of a rational homology sphere is contained in \cite{BC}. (The general case, whose proof is a straightforward
generalization of the previous case, is spelled out in \cite{C}.) For completeness, we briefly recall the idea of the proof.
One chooses a tubular neighborhhood $U$ of $\de C_2(M)$ and a subneighborhood $V$. One picks a compactly supported function $\rho$ which is constant and equal
to $-1$ on $V$ and is even under the action of $T$. Let $p$ be the projection $U\to \de C_2(M)$.
The differential of the form $\rho p^*\eta$ is a representative of the Thom class of the normal bundle of $\Delta$.
The form $\rho p^*\eta$ may be extended by zero to the whole of $C_2(M)$ and its differential is then a representative of the Poincar\'e dual of $\Delta$.
Thus, its difference from the given representative $\chi$ will be exact and actually \textsf{regular} on the diagonal (see Appendix~\ref{app-reg}
for the definition). Namely,
$\ddd\rho p^*\eta=\pi^*(\chi+\ddd\alpha)$, $\alpha\in\Omega^3(M\times M)$. Since both $\rho p^*\eta$ and $\chi$ are odd under the action of $T$,
one may also choose $\alpha$ to be odd. Finally, one sets $\hat\eta:=\rho p^*\eta-\pi^*\alpha$.
\end{proof}

\subsection{The improved propagator}
In the previous Sections it was also important to assume the conditions $K\circ\iota=0$ and $K^2=0$. In terms of the propagator $\hat\eta$
they correspond to additional conditions that are easily expressed by using the

\begin{Def}[Compactification of configuration spaces]
The FMAS compactification $C_n(M)$ of the configuration space
$C_n^0(M):=\{(x_1,\dots,x_n)\in M^n : x_i\not=x_j\ \forall i\not=j\}$ is obtained by taking the closure
\[
C_n(M) := \overline{C_n^0(M)}\subset M^n\times\prod_{S\subset\{1,\dots,n\}:|S|\ge2}
\Bl(M^S,\Delta_S),
\]
where $\Bl(M^S,\Delta_S)$ denotes the differential-geometric blowup obtained by replacing the principal diagonal $\Delta_S$ in $M^S$ with
its unit normal bundle $N(\Delta_S)/\mathbb{R}^+_*$. See \cite{AS,BT}.
\end{Def}
We recall that the $C_n(M)$ are smooth manifolds with corners.

\begin{Not}\label{not-pi}
Given a differential form $\gamma$ on $M$, we write $\gamma_i:=\pi_i^*\gamma$ where $\pi_i$ is the extension to $C_n(M)$ of the
projection $C_n^0(M)\to M$, $(x_1,\dots,x_n)\mapsto x_i$. We also set $\hat\eta_{ij}:=\pi_{ij}^*\hat\eta$,
where $\pi_{ij}$ 
is the extension to the compactifications of the projection $C_n^0(M)\to C_2^0(M)$,
$(x_1,\dots,x_n)\mapsto (x_i,x_j)$. 
Given a product of such forms on a compactified configuration space and a set of indices $i,j,\ldots$, we denote by
$\int_{i,j,\ldots}$ the fiber integral over the points $x_i,x_j,\ldots$. This notation also takes care of orientation. Namely,
if $i,j,\ldots$ are not odered, the integral carries the sign of the permutation to order them.
\end{Not}

We are finally ready to list the two simple properties corresponding to $K\circ\iota=0$ and $K^2=0$:
\begin{description}
\item[P4] $\int_2\hat\eta_{12}\,\gamma_2=0$ for all $\gamma\in\iota(H^\bullet(M))$.
\item[P5] $\int_2 \hat\eta_{12}\,\hat\eta_{23}=0$.
\end{description}

\begin{Lem}\label{l:P1234}
For any given odd global angular form $\eta$ on $STM$, there exists a propagator $\hat\eta\in\Omega^2(C_2(M))$ satisfying
\textsf{P1, P2', P3, P4}.
\end{Lem}

\begin{proof}
The idea is to apply transformation \eqref{K_2}
to a propagator as the one constructed in Lemma~\ref{l:P123}. This must however be reformulated in terms of integral kernels.
Namely, let $\hat\eta$ satisfy \textsf{P1,P2',P3}. Set
\begin{multline*}
\lambda_{12}:=\int_3\eta_{13}\,v_3-\int_3\eta_{23}v_3+
\int_3\eta_{13}\alpha_{3,i}\beta_2^i-\beta_1^i\int_3\eta_{23}\alpha_{3,i}+
\int_3\eta_{13}\beta_3^i\alpha_{2,i}+\\ +\alpha_{1,i}\int_3\eta_{23}\beta_3^i
-\alpha_{1,i}\int_{3,4}\beta_3^i\eta_{34}\beta_4^j\alpha_{2,j}
-\beta_1^i\int_{3,4}\alpha_{3,i}\eta_{34}v_4
-\int_{3,4}v_3\eta_{34}\alpha_{4,i}\beta_2^i.
\end{multline*}
By construction the new propagator $\hat\eta-\lambda$ satisfies \textsf{P3} and \textsf{P4}.
It is not difficult to check that $\lambda$ is closed, so \textsf{P1} is still satisfied.
Finally observe that the integrals on one argument simply produce a form on $M$, while the integrals on two arguments
simply produce a number. As a consequence, $\lambda$ is (the pullback by $\pi$ of) a form on $M\times M$. As it is also $T$\ndash odd,
its restriction to the boundary vanishes. Thus, \textsf{P2'} still holds.
\end{proof}

\subsection{On property \textsf{P5}}
\begin{Lem}\label{l:P12345}
For any given odd global angular form $\eta$ on $STM$, there exists a propagator $\hat\eta\in\Omega^2(C_2(M))$ satisfying
\textsf{P1, P2, P3, P4, P5}. 
\end{Lem}
\begin{proof}
We apply transformation \eqref{K_3=KdK} in the equivalent form
\[
K_3=K_2+[\ddd,\ddd K_2^3]
\]
to the propagator $\hat\eta$ constructed in Lemma~\ref{l:P1234}. In terms of integral kernels, the new propagator
is $\hat\eta+\gamma$ with $\gamma_{12}:=\ddd\ddd_1 f_{12}$ and
\[
f_{12}:=\int_{3,4}\hat\eta_{13}\hat\eta_{34}\hat\eta_{42}.
\]
Properties \textsf{P1} and \textsf{P2} are obviously satisfied as we have changed the propagator by an exact term.
As for property \textsf{P3}, observe that equivalently $\gamma$ may be written as $\ddd_2\ddd_1 f_{12}$ and that $f$ is even under the action of $T$.
Finally property \textsf{P4} is easily checked by integration.
\end{proof}
It would be very useful to prove the following
\begin{conj}\label{conj}
The propagator constructed in Lemma~\ref{l:P12345} also satisfies property \textsf{P2'}.
\end{conj}
Observe that since $\gamma$ is $T$\ndash odd, it would suffice to show that it is regular on the diagonal (i.e., a pullback from $M\times M$).
For this it would suffice to show that $f$ or at least $\ddd f$ has this property. By its definition $f$ looks rather regular, but at the moment
we have no complete proof of this fact.

\begin{Rem}[The Riemannian propagator]\label{r:Riem}
The physicists' treatment of Chern--Simons theory would simply be to choose a Riemannian metric $g$ and use it to impose the Lorentz gauge-fixing.
Out of this one gets the propagator $\ddd^*\circ G$, where $G$ is the Green function for $\Delta+\mathcal{P}'$, where $\Delta$ is the Laplace operator and
$\mathcal{P}'$ is the projection to harmonic forms in the Hodge decomposition. The integral kernel of this propagator is a smooth two-form $\hat\eta$
on $C_2(M)$ satisfying properties \textsf{P1, P2', P3, P4, P5} with odd angular form on the boundary of the type described in
Remark~\ref{r:BC}. In this case, the metric connection is actually the Levi-Civita connection for the chosen Riemannian metric.
See \cite{AS} and \cite[Remark 3.6]{BC}.
\end{Rem}


\subsection{The construction of the invariant by a framed propagator}
We now fix the boundary value of the propagator as in Remark~\ref{r:K} for a given choice $f$ of framing. We also choose an embedding $\iota$
of cohomology and
denote by $P_{\iota,f}$ the space
of propagators satisfying properties \textsf{P1,P2',P3}. Observe that, by Lemma~\ref{l:P1234}, $P_{\iota,f}$ is not empty.
For $\hat\eta\in P_{\iota,f}$ we define $\mr{Ind}_{\iota,\hat\eta}\in \Fun(\FF')[[\hbar]]$, $\FF'=H^\bullet(M,\frg)[1]=H^\bullet(M)[1]\otimes\frg$,
analogously to $\mr{Ind}_{\iota,K}$ as at the beginning
of subsection~\ref{s:relax} by the following obvious changes of notations:
\begin{enumerate}
\item Every chain homotopy $K$ is replaced by a propagator $\hat\eta$;
\item every vertex is replaced by a point in the compactified configuration space over which we eventually integrate.
\end{enumerate}
Signs may be taken care of by choosing an ordering of vertices and of half edges at each vertex (see, e.g., \cite{BC}).

All computations in Section~\ref{s:toy} go through as they are simply replaced by Stokes theorem:
\[
\ddd\int_{C_n(M)} = (-1)^n\int_{C_n(M)}\ddd + \int_{\de C_n(M)}.
\]
Among the codimension\ndash one boundary components of $C_n(M)$ we distinguish between principal and hidden faces: the former correspond to the collapse of
exactly two points, the latter to the collapse of more than two points. Principal faces contribute by the same combinatorics as in Section~\ref{s:toy},
whereas hidden faces do not contribute by our choice of propagators because of Kontsevich's vanishing Lemmata \cite{K}.
As a result we conclude that $\mr{Ind}_{\iota,\hat\eta}$ satisfies the quantum master equation for every choice of induction data as above.
For more details, we refer to Appendix~\ref{app-details}.

By the same reasoning and by the arguments
of subsection~\ref{s:relax}, we may prove that $\mr{Ind}_{\iota,\hat\eta}$ is canonically equivalent to a strict effective action. This allows us to recover at least the two\ndash loop part of the complete invariant of a rational homology sphere
as discussed in subsection~\ref{invreleffact}.

\begin{Rem}
If Conjecture~\ref{conj} were true, the space $P_{\iota,f}'$ of propagators satisfying properties \textsf{P1,P2',P3,P4,P5} would not be empty.
We could then repeat the above construction using $P_{\iota,f}'$ instead of $P_{\iota,f}$ and get a strict effective action directly.
The discussions of subsections~\ref{sec 2.2}, \ref{sec 2.3} and \ref{sec 2.4} would then go through and, in particular,
Propositions~\ref{prop4}, \ref{prop5}, \ref{prop6}, \ref{prop7} would hold.
\end{Rem}

\begin{Rem}
In \cite{BC} an invariant for framed rational homology spheres was introduced. The boundary condition for the propagator was different (see next subsection),
but this is immaterial for the present discussion. Namely, choose a propagator in $P_{\iota,f}$ and define
$\tilde\eta_{123}:=\hat\eta_{12}+\hat\eta_{23}+\hat\eta_{31}$. If $M$ is a rational homology sphere, $\tilde\eta_{123}$ is closed.
Now take the graphs appearing in the constant part of the strict effective action and reinterpret them as follows:
\begin{enumerate}
\item Each vertex is replaced by a point in the compactified configuration space;
\item an extra point $x_0$ is added on which one puts the representative $v\in\iota(H^3(M))$ with $\int v=1$;
\item each chain homotopy is replaced by $\tilde\eta$ (more precisely, the chain homotopy between vertices $i$ and $j$ is replaced
by $\tilde\eta_{ij0}$).
\end{enumerate}
It is now possible to show that this produces an invariant of $(M,f)$. This is a different way of getting the invariant
corresponding to the constant part of the strict effective action
for a choice of propagator
not necessarily satisfying property \textsf{P5}. We do not have a direct proof that this invariant is the same. The indirect proof consists
of showing that both invariants are finite type with the same normalizations along the lines of \cite{KT}.

If Conjecture~\ref{conj} were true, then it would immediately follow that this invariant is exactly the constant part of the strict effective action.
In fact, for a propagator
as in the Conjecture, it is not difficult to show that only the term $\hat\eta_{ij}$ in each $\tilde\eta_{ij0}$ would contribute (and the integration
over $x_0$ would then decouple).
Since the induced action is constant
on $P_{\iota,f}$, by restriction to $P_{\iota,f}'\subset P_{\iota,f}$ we would prove the claim.
\end{Rem}

\subsection{The unframed propagator}\label{s-unframed}
Instead of using Kontsevich's propagator, one may proceed as in \cite{BC} and define the propagator by choosing the global angular
form on $\de C_2(M)$ as in Remark~\ref{r:BC}. Recall that in this case no choice of framing is required. On the other hand,
one needs to specify a Riemannian metric $g$ and a metric connection $\theta$. We denote by $P_{\iota,g,\theta}$ the space of propagators
corresponding to these choices.

We proceed exactly as in the previous subsection to define the effective action  (see Appendix~\ref{app-details}
for more details).
In particular we want to check independence on the induction data; so we choose a path
in $P_{\iota,g,\theta}$ and consider the effective action $W$ as a function on the shifted cohomology tensor the differential forms on $[0,1]$ and check
whether the extended QME $(\ddd+\hbar\Delta) W + \{W,W\}/2=0$ holds.
The only difference with respect to the previous subsection
is that there is an extra set of boundary components of the configuration
spaces that may appear: namely, the most degenerate faces corresponding to the collapse of all points. These faces may be treated exactly as in
\cite{AS,BC} and one shows that their contribution is a multiple of the first Pontryagin form $-\operatorname{tr}F_\theta^2/(8\pi^2)$, where $F_\theta$
is the curvature of $\theta$. The important point is that the coefficient depends only on the graph involved but not on the $3$\ndash manifold $M$.
As a result the effective action might not satisfy the extended quantum master equation. However, one may easily compensate for this by adding to it
the integral over $M$ of the Chern--Simons $3$\ndash form of the connection $\theta$ pulled back from $F(M)$ to $M$ by choosing a section $f$ (i.e., a framing).
The framing now appears because of this correction but is not present in the propagator.

The main disadvantage of this approach is that one does not know how to compute the universal coefficients. (It is known that the coefficients vanish
for graphs with an odd number of loops, while for the graph with two loops one may compute the coefficient explicitly and see that it is not zero.)
The advantage is that $P_{\iota,g,\theta}$ contains a subspace of propagators satisfying also property \textsf{P5}: These are the integral kernels
constructed in \cite{AS}, see Remark~\ref{r:Riem}.
With these choices, and the addition of the frame-dependent constant as in the previous paragraph, one gets an induced effective action
satisfying all properties stated in subsections~\ref{sec 2.2}, \ref{sec 2.3} and \ref{sec 2.4}.
More precisely, let
\[
\operatorname{CS}(M,\theta,f):= -\frac1{8\pi^2}\int_M f^*\operatorname{Tr}\left(\theta\,\ddd\theta+\frac23\theta^3\right)
\]
be the Chern--Simons integral for a connection $\theta$ on the frame bundle $F(M)$ of $M$ and a framing $f$ (regarded here as a section of $F(M)$).
\begin{thm}\label{thetheorem}
Let $M$ be a compact $3$\ndash manifold and $\frg$ a quadratic Lie algebra. Then
\begin{enumerate}
\item For every choice of Riemannian metric $g$ on $M$, the effective action $W$ constructed using the Riemannian propagator of Remark~\ref{r:Riem}
is a function on $H^\bullet(M,\frg)[1]$, solves the quantum master equation and has the properties described in Proposition~\ref{prop4}.
In addition it has the form given in \eqref{prop6 W} in case $B_1(M)=0$ and in \eqref{prop7 W} in case $M$ is formal.
\item There is a universal element $\phi\in\hbar^2\mathbb{R}[[\hbar^2]]$,
depending only on the choice of Lie algebra $\frg$,
such that
the modified effective action
\[
\widetilde W(M,g,f):=W(M,g)+\phi \operatorname{CS}(M,\theta_g,f),
\]
where $\theta_g$ is the Levi-Civita connection for $g$, solves the QME and is independent of $g$ modulo canonical transformations as in Proposition~\ref{prop5}.
In particular we get invariants for the framed $3$\ndash manifold $(M,f)$ as in Propositions~\ref{prop6} and \ref{prop7}.
\end{enumerate}

\end{thm}

\begin{Rem}
The leading contribution
to $\phi$ may be explicitly computed and yields
\[ \phi=C_2(\frg)\hbar^2/48 + O(\hbar^4)
\]
with $C_2(\frg)=f_{abc}f_{abc}$, where the $f_{abc}$ are the structure constants of $\frg$ in
an orthonormal basis. It is not known whether there are nonvanishing higher order corrections.
\end{Rem}

\appendix

\newcommand{\DDD}{\operatorname{{D}}}
\section{The Chern--Simons manifold invariant}\label{app-details}
In this Appendix we give more details on the construction outlined in subsection~\ref{s-unframed}.
To a graph $\Gamma$ with $|\Gamma|$ vertices we associate an element $\omega_\Gamma$
of $\Omega^\bullet(C_{|\Gamma|}(M))\otimes \Fun(\FF')$ as follows:\footnote{Here and in the rest of the Appendix, the symbol $\otimes$ is understood as the
completed tensor product: i.e., the space of  functions on the Cartesian product of the corresponding supermanifolds.}
\begin{itemize}
\item to each edge we associate the pullback of a propagator by the corresponding projection from $C_{|\Gamma|}(M)$ to $C_2(M)$;
\item to each leaf we associate the pullback (by the corresponding projection from $C_{|\Gamma|}(M)$ to $M$) of $\sum z_\mu^a\gamma^\mu e_a$,
where $\{\gamma^\mu\}$ is the chosen
basis of $H^\bullet(M)$, $\{e_a\}$ is an orthonormal basis of $\frg$, and the $\{z^a_\mu\}$s
are the corresponding coordinate functions on $\FF'$;
\item on each vertex we put a structure constant in the orthonormal basis chosen above.
\end{itemize}
Then we take the wedge product of the differential forms and sum over Lie algebra indices for each edge.
The result does not depend on the choice of orthonormal basis for $\frg$ but
depends on a choice of numbering of the vertices and of orientation of the edges.
If we however make the same choice also to orient
$C_{|\Gamma|}(M)$, then
\[
\int_{C_{|\Gamma|}(M)}\omega_\Gamma\in  \Fun(\FF')
\]
is well defined.
We define $Z$ (the exponential of $W$) to be the sum of $\omega_{|\Gamma|}/|\operatorname{aut}\Gamma|$ over all
trivalent graphs, where $|\operatorname{aut}\Gamma|$ is the order or the group of automorphisms of $\Gamma$.
Proving the QME for $W$ is equivalent to proving $\Delta Z=0$.

The main observation is that Property~\textsf{P1}
of the propagator and the same combinatorics as in the toy model
imply that $\Delta Z$ is obtained by replacing the $\omega_\Gamma$s by $\ddd\omega_\Gamma$s one by one.
We then use Stokes theorem. The contributions of principal faces
(i.e., boundary faces of configuration spaces corresponding to the collapse of two vertices)
sum up to zero thanks 
to the Lie algebra contributions (this is also the same combinatorics as in the toy model).
Hidden faces (i.e., the other boundary faces) may in principle contribute. Let $\gamma$ be the subgraph corresponding
to a hidden face (i.e., the vertices of $\gamma$ are those that collapse and its edges are the edges between such vertices).
By simple dimensional reasons, the hidden faces corresponding to $\gamma$ vanishes if $\gamma$ has a univalent vertex;
if $\gamma$ has a bivalent vertex, its contribution also vanishes by Kontsevich's lemma thanks to Property~\textsf{P3}.
Since we only consider trivalent graphs, we are left with contributions coming from the collapse of \emph{all}\/ vertices
of a connected component of $\Gamma$s with \emph{no leaves}.
These are the graphs that contribute to the constant part of $Z$ and thus of $W$.
The latter contributions also vanish by a simple dimensional argument.

More generally, to keep track of the choices
involved in the propagator, we consider a one parameter family of choices with parameter $t\in I:=[0,1]$
and show $\Delta' \Tilde Z=0$ with $\Delta':=\Delta + \ddd t\,\frac\ddd{\ddd t}$ and
$\Tilde Z\in Fun(\FF')\otimes\Omega^\bullet(I)$ constructed as follows.
Let $\hat\eta$ be a one-parameter family of propagators regarded as an element of
$\Omega^2(C_2(M)\times I)$ related, at every $t\in I$, by Property~\textsf{P1} to the one parameter family $\{\gamma^\mu\}$
of bases of  $\Omega^\bullet(M)$.
Let $\dot\eta\in\Omega^2(C_2(M)\times I)$ and $\dot\gamma^\mu\in\Omega^\bullet(M\times I)$
be their $t$\ndash derivatives.
We may assume $\int_M \gamma^\mu\dot\gamma^\nu=0$, $\forall \mu,\nu$, for more
general choices may be compensated by a linear transformation of $H^\bullet(M)$.
Observe that by Property~\textsf{P2'}, the restriction of $\hat\eta$ to the boundary is fixed, so the restriction
of $\dot\eta$ vanishes. Actually, by construction $\dot\eta$ vanishes in a whole neighborhood of the boundary,
so it is a regular form. Let
\[
\lambda_{13}:=\frac{\int_2\Hat\eta_{12}\,\dot\eta_{23}-\dot\eta_{12}\,\Hat\eta_{23}}2,
\]
which is regular by Lemma~\ref{l-reg} in Appendix~\ref{app-reg}. Also set $\xi^\mu_1:=(\int\Hat\eta_{12}\dot\gamma^\mu_2)$.
We define
\begin{align*}
\Tilde\eta&:= \Hat\eta-\lambda\,\ddd t,\\
\Tilde\gamma^\mu&:=\gamma^\mu+\xi^\mu\,\ddd t,\\
\Tilde\chi &:= g_{\mu\nu}\Tilde\gamma^\mu\,\Tilde\gamma^\nu = \chi + O(\ddd t),
\end{align*}
where $g_{\mu\nu}$ is the metric on $H^\bullet(M)$ in the given basis.
A simple computation then shows $\DDD\Tilde\eta=\Tilde\chi$ and $\DDD\Tilde\gamma^\mu=0$
with $\DDD:=\ddd+\ddd t\frac\ddd{\ddd t}$.
Finally, we define $\Tilde Z$ as above by using $\Tilde\eta$ and $\Tilde\gamma$
instead of $\Hat\eta$ and $\gamma$, respectively. We now observe that applying $\Delta'$ to $\Tilde Z$ is the same
as applying $\DDD$ to the propagators. Reasoning as above by Stokes theorem, we see that the only possible
nonvanishing contributions come from hidden faces corresponding to the collapse of all vertices of a connected component
with \emph{no leaves}.\footnote{Observe that, since $\lambda$ is regular, only $\eta$ will appear in the boundary computations.}
These contributions also vanish if one uses a framed propagator, whereas the choice of
an unframed propagator yields some constant times the integral of the Pontryagin form on $M$ as a one-form on $I$,
see \cite{BC}.

By writing $\Tilde Z= Z + \zeta\ddd t$, we may see that the equivalences are produced by $\zeta$. These are very particular
kinds of BV equivalences as $\zeta$ consists of graphs decorated by propagators and generators of cohomology classes with the
exception of one edge that is decorated by $\lambda$ or one leaf that is decorated by $\xi$.

If Property \textsf{P5} holds, the computation of $\Tilde Z$ simplifies drastically. By construction we
obtain $\int_2\Hat\eta_{12}\lambda_{23}=\int_2\lambda_{12}\Hat\eta_{23}=0$ and $\int_2\eta_{12}\xi^\mu_2=0$.
As a result, whenever a vertex is decorated by $1\in\Omega^0(M)$ the corresponding integral vanishes.
Thus, $\Tilde Z$, as a function on $H^\bullet(M,\frg)[1]$ is independent of the coordinates in degree $1$ apart from
the trivial classical term. Since $Z$ is of degree zero, it will only depend on the coordinates of degree zero.
Since $\zeta$ is of degree $-1$, it will be linear in the coordinates of degree $-1$ with coefficients depending
on the coordinates of degree zero.

\section{Regular forms}\label{app-reg}
Let $\varpi\colon C_n(M)\to M^n$ be the extension to the compactification of the inclusion $C_n(M)^0\to M^n$.
We call a form on $C_n(M)$ \textsf{regular} if it is a pullback by $\varpi$.
Recall the maps $\pi_i$s and $\pi_{ij}$s defined in Notation~\fullref{not-pi}.


\begin{Lem}\label{l-reg}
Let $\alpha$ and $\beta$ be differential forms on $C_2(M)$ and let at least one of them be regular.
Then their convolution $\alpha*\beta:=\int_2\alpha_{12}\,\beta_{23}:=\pi_{13,*}(\pi_{12}^*\alpha\,\pi_{23}^*\beta)$
is regular.
\end{Lem}
\newcommand{\pr}{\operatorname{{pr}}}
\begin{proof}
Suppose that, e.g., $\alpha$ is regular; i.e., $\alpha=\varpi^*\alpha'$, $\alpha'\in\Omega^\bullet(M\times M)$.
Define
\[
\gamma = (\pr_1\times\pi_2)_*\left(
(\pr_1\times\pi_1)^*\alpha'\,\pr_2^*\beta
\right)\in\Omega^\bullet(M\times M),
\]
where $\pr_1$ and $\pr_2$ are the projections from $M\times C_2(M)$ to the two factors.
It then follows that $\alpha*\beta=\varpi^*\gamma$.
\end{proof}


\newcommand\qq{\rm}
\newcommand\cmp[1]{{\qq Commun.\ Math.\ Phys.\ \bf #1}}
\newcommand\jmp[1]{{\qq J.\ Math.\ Phys.\ \bf #1}}
\newcommand\pl[1]{{\qq Phys.\ Lett.\ \bf #1}}
\newcommand\np[1]{{\qq Nucl.\ Phys.\ \bf #1}}
\newcommand\mpl[1]{{\qq Mod.\ Phys.\ Lett.\ \bf #1}}
\newcommand\prl[1]{{\qq Phys.\ Rev.\ Lett.\ \bf #1}}
\newcommand\plb[1]{{\qq Phys.\ Lett.\ \bf B #1}}
\newcommand\npb[1]{{\qq Nucl.\ Phys.\ \bf B #1}}
\newcommand\lmp[1]{{\qq Lett.\ Math.\ Phys.\ \bf #1}}
\newcommand\jsp[1]{{\qq J. Stat.\ Phys.\ \bf #1}}
\newcommand\ijmp[1]{{\qq Int.\ J.\ Mod.\ Phys.\ \bf #1}}
\newcommand\ijm[1]{{\qq Int.\ J.\ Math.\ \bf #1}}
\newcommand\cqg[1]{{\qq Class.\ Quant.\ Grav.\ \bf #1}}
\newcommand\prept[1]{{\qq Phys.\ Rept.\ \bf #1}}
\newcommand\tmp[1]{{\qq Theor.\ Math.\ Phys.\ \bf #1}}
\newcommand\anp[1]{{\qq Ann.\ Phys.\ \bf #1}}
\newcommand\anm[1]{{\qq Ann.\ Math.\ \bf #1}}
\newcommand\man[1]{{\qq Math.\ Ann.\bf #1}}
\newcommand\inm[1]{{\qq Invent.\ Math.\ \bf #1}}
\newcommand\adm[1]{{\qq Adv.\ Math.\ \bf #1}}
\newcommand\asm[1]{{\qq Adv.\ Sov.\ Math.\ \bf #1}}
\newcommand\rms[1]{{\qq Russ.\ Math.\ Surveys \bf #1}}
\newcommand\jp[1]{{\qq J.\ Phys.\ \bf #1}}
\newcommand\bAMS[1]{{\qq Bull.\ Amer.\ Math.\ Soc.\ \bf #1}}
\newcommand\tAMS[1]{{\qq Trans.\ Amer.\ Math.\ Soc.\ \bf #1}}
\newcommand\jdg[1]{{\qq J.\ Diff.\ Geom.\ \bf #1}}
\newcommand\conm[1]{{\qq Cont.\ Math.\ \bf #1}}
\newcommand\Top[1]{{\qq Topology \bf #1}}
\newcommand\mpcps[1]{{\qq Math.\ Proc.\ Camb.\ Phil.\ Soc.\ \bf #1}}
\newcommand\jkt[1]{{\qq J.\ of Knot Theory and Its Ramifications \bf #1}}
\newcommand\selma[1]{{\qq Selecta Math.\ \bf #1}}
\newcommand\phs[1]{{\qq Physica Scripta \bf #1}}
\newcommand\gt[1]{{\qq Geometry \& Topology \bf #1}}
\newcommand\agt[1]{{\qq Algebr.\ Geom.\ Topol.\ \bf #1}}
\newcommand\dmj[1]{{\qq Duke Math.\ J. \bf #1}}
\newcommand\jgp[1]{{\qq J. Geom.\ Phys.\ \bf #1}}

\end{document}